\documentclass[reqno, 11pt]{amsart}
\usepackage{fullpage,amssymb,stmaryrd}
\usepackage{lmodern}
\usepackage[T2A,T1]{fontenc}
\usepackage[utf8]{inputenc}
\usepackage[russian,english]{babel}
\usepackage{hyperref}
\usepackage[hyphenbreaks]{xurl}
\usepackage{amsmath}
\usepackage{xcolor}
\usepackage{xparse}
\usepackage{amsthm}
\usepackage{pgfplots}
\pgfplotsset{compat=1.18}
\usepackage{mathrsfs}
\usepackage[autostyle, english = american]{csquotes}
\usepackage[backend=biber,style=alphabetic,giveninits=true,maxnames=50,autolang=other]{biblatex}
\usepackage{tikz-cd}
\usepackage[shortlabels]{enumitem}
\hypersetup{
pdftitle={The Tate conjecture for abelian fourfolds over finite fields},
pdfsubject={Mathematics, Algebraic Geometry, Arithmetic Geometry},
pdfauthor={Matt Broe},
pdfkeywords={},
breaklinks=true,
hidelinks
}
\newtheoremstyle{style}
{} % Space above
{} % Space below
{\itshape} % Body font
{} % Indent amount
{\bfseries} % Theorem head font
{.} % Punctuation after theorem head
{.5em} % Space after theorem head
{} % Theorem head spec (can be left empty, meaning `normal')

\newtheoremstyle{theoremnum}
{} % Space above
{} % Space below
{\itshape} % Body font
{} % Indent amount
{\bfseries} % Theorem head font
{.} % Punctuation after theorem head
{.5em} % Space after theorem head
{\thmname{#1}\thmnote{ \bfseries #3}}%%% Thm head spec

\theoremstyle{style}
\newcommand{\comment}[1]{}
\newtheorem{thm}{Theorem}[section]
\newtheorem{cor}[thm]{Corollary}
\newtheorem{prop}[thm]{Proposition}
\newtheorem{lem}[thm]{Lemma}

\newtheorem{asm}[thm]{Assumption}

\newtheorem*{thm*}{Theorem}
\newtheorem*{cor*}{Corollary}
\newtheorem*{prop*}{Proposition}
\newtheorem*{lem*}{Lemma}
\newtheorem*{conj*}{Conjecture}
\newtheorem*{quest*}{Question}
\newtheorem*{claim*}{Claim}
\newtheorem*{ppty*}{Property}

\theoremstyle{theoremnum}

\theoremstyle{definition}
\newtheorem{defn}[thm]{Definition}

\newtheorem{nota}[thm]{Notation}

\theoremstyle{remark}
\newtheorem{rem}[thm]{Remark}
\newtheorem*{rem*}{Remark}

\makeatletter
\def\eqref{\@ifstar\@eqref\@@eqref}
\def\@eqref#1{\textup{\tagform@{\ref*{#1}}}}
\def\@@eqref#1{\textup{\tagform@{\ref{#1}}}}
\makeatother

\DeclareMathOperator{\Spec}{Spec}
\DeclareMathOperator{\Gal}{Gal}

\DeclareMathOperator{\Aut}{Aut}

\DeclareMathOperator{\Hom}{Hom}
\DeclareMathOperator{\Sym}{Sym}

\DeclareMathOperator{\End}{End}

\DeclareMathOperator{\CH}{CH}

\DeclareMathOperator{\rk}{rk}

\DeclareMathOperator{\Frob}{Frob}

\DeclareMathOperator{\rat}{rat}

\DeclareMathOperator{\num}{num}

\DeclareMathOperator{\et}{\acute{e}t}

\DeclareMathOperator{\gr}{gr}

\DeclareMathOperator{\Fil}{Fil}
\DeclareMathOperator{\MOD}{MOD}
\DeclareMathOperator{\HK}{HK}
\DeclareMathOperator{\dR}{dR}
\DeclareMathOperator{\cris}{cris}
\DeclareMathOperator{\Rep}{Rep}
\DeclareMathOperator{\Hdg}{Hdg}

\newcommand{\Z}{{\mathbb{Z}}}

\newcommand{\F}{{\mathbb{F}}}

\newcommand{\Q}{{\mathbb{Q}}}

\newcommand{\Co}{{\mathbb{C}}}

\newcommand{\mf}[1]{\mathfrak{#1}}
\newcommand{\mc}[1]{\mathcal{#1}}

\newcommand{\ol}[1]{\overline{#1}}

\numberwithin{equation}{section}

\title{The Tate conjecture for abelian fourfolds over finite fields}

\author{Matt Broe}

\begin{document}
\linespread{1.3}\selectfont
\begin{abstract}
    We prove the Tate conjecture for abelian fourfolds over finite fields. This is the first resolution of the conjecture for all abelian varieties of a fixed dimension over finite fields since the work of Tate in the 1960s. The proof relies on techniques from Ancona's proof of the standard conjecture of Hodge type for abelian fourfolds, and ultimately reduces to Markman's results on the algebraicity of Weil classes on complex abelian varieties.
    \par
    Combining the above case of the Tate conjecture with theorems of Ancona and Kahn, we deduce that the standard conjecture on homological versus numerical equivalence holds for abelian fourfolds over arbitrary fields. This completes the proof of the standard conjectures for abelian fourfolds. 
\end{abstract}

    \maketitle
    \tableofcontents

    \section{Introduction}
    In this paper we prove the Tate conjecture for abelian fourfolds over finite fields \cite[Conjecture 1]{Tate1965}. Let $k \cong \F_q$ be a finite field of characteristic $p$, and let $\ell$ be a prime invertible in $k$. Let $G_k$ be the absolute Galois group of $k$.
    \begin{thm}\label{thmTate}
        For every abelian fourfold $A$ over $k$, the $\ell$-adic cycle class map
        \begin{gather*}
            \CH^2(A)_{\Q_\ell} \to H^4_{\et}(A_{\ol{k}}, \Q_\ell(2))^{G_k}
        \end{gather*}
        is surjective.
    \end{thm}
    Tate proved his conjecture for divisors on $A$ \cite{Tate1966}; the conjecture for codimension-three cycles on $A$ follows from the case of divisors, by the classical hard Lefschetz theorem for abelian varieties. Theorem \ref{thmTate} thus settles the Tate conjecture for $A$ in all degrees.
    \par
    Kahn showed that the Tate conjecture for $A$ implies the agreement of rational and numerical equivalence for algebraic cycles with rational coefficients on $A$ \cite[Th\'eor\`eme 1]{Kahn2003}. Combining this with Ancona's proof of the standard conjecture of Hodge type for abelian fourfolds \cite[Theorem 3.18 and Corollary 3.19]{Ancona2021}, we obtain the following corollary of Theorem \ref{thmTate}.
    \begin{cor}\label{corStdIntro}
        Grothendieck's standard conjecture $D$ \cite{Grothendieck1969} holds for every abelian fourfold $A'$ over an arbitrary field $k'$. Namely, if $\ell$ is invertible in $k'$, then $\ell$-adic homological equivalence and numerical equivalence coincide on $A'$.
    \end{cor}
    Together with the results of Lieberman \cite[Theorem 2A11]{Kleiman1968} and Ancona, Corollary \ref{corStdIntro} completes the proof of all of the standard conjectures for abelian fourfolds. It also strengthens Clozel's theorem, which states that for an abelian variety $V$ over the finite field $k$, $\ell$-adic homological and numerical equivalence coincide on $V$ for a set of primes $\ell$ of positive density \cite[Th\'eor\`eme 1 and Corollaire 1]{Clozel1999}. 

    \begin{rem}\label{remStdHodge}
        Here \enquote{the standard conjecture of Hodge type for $X$} refers to a variation on the conjecture denoted by $\Hdg(X)$ in \cite{Grothendieck1969}. The former conjecture concerns cycles modulo numerical equivalence, while Grothendieck's original formulation concerns cycles modulo homological equivalence. In the presence of the standard conjecture $D$ for $X$, there is no distinction between the two formulations. Note also that in characteristic zero, the conjecture $\Hdg(X)$ is classically known for all smooth projective varieties. 
        \par
        The original conjecture $\Hdg(X)$, together with the standard conjecture of Lefschetz type for $X$, implies the standard conjecture $D$ for $X$. However, this implication is not preserved when $\Hdg(X)$ is replaced by the standard conjecture of Hodge type in the above sense. 
    \end{rem}

    \subsection*{Prior work on the Tate and Hodge conjectures}
    
    Abelian fourfolds have long been a focus of research on the Tate and Hodge conjectures. Pohlmann \cite{Pohlmann1968} presented an example, due to Mumford, of an abelian fourfold over $\Co$ whose Betti cohomology contained exceptional Hodge classes, i.e. Hodge classes which are not contained in the span of products of divisor classes. Weil \cite{Weil1979} generalized Mumford's example by defining what are now called abelian varieties of Weil type, which generically possess certain exceptional Hodge classes, known as Weil classes. Results of Tankeev \cite{Tankeev1982} and Moonen--Zarhin \cite{MZ1999} later reduced the Hodge conjecture for all abelian varieties of dimension $\le 5$ to the algebraicity of Weil classes on abelian fourfolds of Weil type. The latter statement was finally proved by Markman in \cite[Corollary 1.6.1]{Markman2025}.  
    \par
    Milne showed that the Hodge conjecture for CM abelian varieties implies the Tate conjecture for abelian varieties over finite fields \cite{Milne1999b}. By reworking the method of Lefschetz groups used in loc. cit., he later obtained unconditional cases of the Tate conjecture \cite{Milne2022}. However, we do not know whether it is possible to prove the statement of Theorem \ref{thmTate} solely by combining Markman's results with arguments based on Lefschetz groups. See also \cite{APFV2025} for further background and results on the Tate conjecture for abelian varieties over finite fields. 

    \subsection*{Sketch of the proof of Theorem \ref{thmTate}}
    As mentioned earlier, Ancona \cite{Ancona2021} established the standard conjecture of Hodge type for abelian fourfolds, using tools from $p$-adic Hodge theory and the theory of Chow motives. A large part of the proof of Theorem \ref{thmTate} draws on ideas from Ancona's work. 
    \par
    Given an abelian fourfold $A$ over $k$ (subject to the technical Assumption \ref{asmExt}, which does not affect the generality of the argument), we use the extra endomorphisms of $A$ to decompose the Chow motive of $A$ into smaller summands. The exceptional Tate classes of $A$ are shown to be supported on certain rank-two summands, corresponding to imaginary quadratic subfields $K \subseteq \End^0(A)$ (Corollary \ref{corExoticSubWeil}). We denote such a summand by $\mc{W}_K(A)$, and refer to it as an exceptional Weil motive. The argument which confines the exceptional classes to these summands is mostly adapted from \cite[Section 7]{Ancona2021}, but requires some care to remove hypotheses in loc. cit. on the existence of algebraic classes.
    \par
    Fix an isomorphism $\ol{\Q}_p \cong \Co$. Up to isogeny, we may lift $A$, together with the decomposition of its motive, to a CM abelian fourfold $\mc{A}_\Co$ over the complex numbers. We write $\mc{W}_K(\mc{A}_\Co)$ for the summand of the motive of $\mc{A}_\Co$ which specializes to $\mc{W}_K(A)$. One may hope to use Markman's results to show that the Betti cohomology of $\mc{W}_K(\mc{A}_\Co)(2)$ is spanned by algebraic classes, and then specialize to deduce that the $\ell$-adic cohomology of $\mc{W}_K(A)(2)$ is similarly spanned by algebraic classes. If $H^4_B(\mc{W}_K(\mc{A}_\Co), \Co)$ has pure Hodge type $(2, 2)$, then this approach goes through as described. However, we must also treat the cases where $H^4_B(\mc{W}_K(\mc{A}_\Co), \Co)$ has Hodge type $(4, 0) + (0, 4)$, or $(3,1) + (1,3)$. In the former case, $p$-adic Hodge-theoretic arguments of Ancona show that $A$ is supersingular, in which case the Tate conjecture for $A$ is well-known. 
    \par
    In the case of Hodge type $(3, 1) + (1, 3)$, some more work is required. One cannot in general reduce from this case to the previous two: indeed, Ancona constructed examples where $H^4_B(\mc{W}_K(\mc{A}_\Co), \Co)$ has Hodge type $(3, 1) + (1, 3)$ for \textit{every} choice of CM lift $\mc{A}_\Co$ of $A$ obtained as above \cite[Proposition A.1]{Ancona2021}. Instead, we again use $p$-adic Hodge theory to show that there is an elliptic curve $E$ over $\ol{\Q}_p$ with complex multiplication by $K$ and supersingular reduction (Lemma \ref{lemPNonSplit}). Then $\mc{A}_\Co \times_\Co E^2$ can be equipped with a polarization $\lambda$ whose determinant is equal to $-1$, up to multiplication by the norm of an element of $K^\times$ \cite[Proposition 1.6]{Milne2022}. In other words, $(\mc{A}_{\Co} \times_\Co E^2, K, \lambda)$ is a polarized abelian sixfold of split Weil type. Markman proved the algebraicity of Weil classes on such sixfolds \cite[Theorem 1.5.1]{Markman2025}, and via specialization we find that the $\ell$-adic cohomology of $\mc{W}_K(A_{\ol{k}}) \otimes_K \mc{W}_K(E_{\ol{k}}^2)(3)$ is spanned by algebraic classes. But $\mc{W}_K(E_{\ol{k}}^2)(1)$ is the unit object in the category of motives over $\ol{k}$ with $K$-coefficients, so we conclude that the $\ell$-adic cohomology of $\mc{W}_K(A_{\ol{k}})(2)$ is spanned by algebraic classes. Thus the exceptional Tate classes on $A$ are algebraic, which implies Theorem \ref{thmTate}.

    \begin{rem}\label{remStabilize}
        The idea of taking a product of $\mc{A}_\Co$ with a power of a suitable CM elliptic curve to obtain a higher-dimensional abelian variety of Weil type has a precedent in \cite[Proposition 1.10]{Milne2022}. It can also be viewed as related to a technique used in Markman's proof of the algebraicity of Weil classes on a complex abelian fourfold $V$ of Weil type. Markman first proves the algebraicity of Weil classes for abelian sixfolds of split Weil type \cite[Theorem 1.5.1]{Markman2025}, then notes that there exists a suitable abelian surface $S$ such that $V \times_\Co S$ is of split Weil type. He then deduces the desired algebraicity statement for $V$ from a result of Schoen \cite[Proposition 10]{Schoen1998}. See \cite[Section 11.5]{Markman2026} for background. 
    \end{rem}

    \begin{rem}
        Note that the preprint \cite[Remark 1.18]{Jiang2025} also indicates an argument to reduce certain cases of the Tate conjecture for abelian fourfolds in positive characteristic to Markman's results, including the case of ordinary abelian fourfolds.
    \end{rem}

    \subsection*{Structure of the paper}
    In Section \ref{secNotation} we fix some global notation and conventions. In Section \ref{secDecomp} we review the decomposition of a motive induced by the action of a finite \'etale algebra, and define Weil motives. In Section \ref{secCmStruct} we introduce CM-structures, which are certain maximal commutative subalgebras of $\End^0(A)$, for $A$ an abelian variety over a finite field. We also describe the decomposition of the motive of $A$ induced by a CM-structure. In Section \ref{secRedWeilMotive} we show that the exceptional Tate classes on an abelian fourfold $A$ over a finite field are supported on certain Weil motives. In Section \ref{secLift} we recall that $A$ can be lifted (up to isogeny) to a CM abelian variety in characteristic zero, and explain the specialization of algebraic classes from characteristic zero to characteristic $p$. In Section \ref{secProofThmTate} we complete the proof of Theorem \ref{thmTate}. In Section \ref{secStd} we present some consequences of Theorem \ref{thmTate}, including the application to the standard conjecture $D$. In Appendix \ref{appHK} we collect properties of Hyodo--Kato cohomology which are needed elsewhere in the paper.

    \section{Notation and conventions}\label{secNotation}
    In the body of the paper, the letter $k$ will denote a field, arbitrary unless otherwise stated. We let $\ol{k}$ denote a fixed algebraic closure of $k$, and let $G_k = \Aut(\ol{k} / k)$ denote the absolute Galois group of $k$. The letter $p$ will denote a fixed prime, and $q = p^f$ will denote a positive integer power of $p$.
    \subsection{Chow motives}
    We refer to \cite{Andre2004} for background on Chow motives (see \cite[Tag 0FG9]{stacks-project} for a useful quick reference). Our conventions for tensor categories are as in \cite{Andre2005}, which also serves as a reference for facts about finite-dimensional motives, in the sense of Kimura \cite{Kimura2005}.
    \par
    The category of \textbf{motives} modulo rational equivalence over a field $k$, with coefficients in a characteristic zero field $L$, is denoted by $\mc{M}(k)_L$. We write $\mf{h}(X)_L \in \mc{M}(k)_L$ for the motive of a smooth projective variety $X$ over $k$. The monoidal unit in this category is denoted by $\mathbf{1}$ (or $\mathbf{1}_L$, when we want to make the coefficient field explicit). For a motive $M \in \mc{M}(k)_L$, and $r \in \Z$, the \textbf{Chow group} $\CH^r(M)_L$ is defined to be $\Hom_{\mc{M}(k)_L}(\mathbf{1}(-r), M)$. 
    \par
    The space of \textbf{algebraic cycles} in $M$ is defined to be $\CH^0(M)_L$. Given a realization functor $R$ on $\mc{M}(k)_L$ induced by a Weil cohomology theory with coefficients in a characteristic zero field $F$, the \textbf{algebraic classes} in $R(M)$ are defined to be the elements of the $F$-linear span of the image of $\CH^0(M)_L$ under $R$. When the classical realizations of $M$ \cite[\S3.4]{Andre2004} are each concentrated in a single degree $d$, the \textbf{rank} $\rk M$ of $M$ is defined as the dimension of any classical realization of $M$ \cite[Definition 4.3]{Ancona2021}. This definition is independent of the choice of realization by \cite[Th\'eor\`eme 4.2.5.2]{Andre2004}. If $\rk' M$ denotes the rank of $M$ in the sense of \cite{Andre2005}, then $\rk M = (-1)^d \rk' M$.
    \par
    If $L' / L$ is a finite extension, we may implicitly consider $\mc{M}(k)_{L'}$ as the category of $L'$-module objects in $\mc{M}(k)_L$ when convenient. In particular, there is a canonical forgetful functor $\mc{M}(k)_{L'} \to \mc{M}(k)_L$. For $M \in \mc{M}(k)_L$, we let $M_{L'} = M \otimes_L L' \in \mc{M}(k)_{L'}$ denote its scalar extension to $L'$. When $L = \Q$, we may sometimes omit the subscript $L$ on all notation related to motives. 

    \subsection{Relative motives}
    For a discrete valuation ring $R$, we similarly consider the category of relative Chow motives $\mc{M}(R)_L$ over $R$ (as in \cite[Section 5.1]{OSullivan2011}), and define associated notation analogously to the case of motives over a field. 
    
    \subsection{Artin motives and motives of abelian type}
    Let $L$ be a field of characteristic zero, and let $\Rep_L G_k$ denote the category of continuous $G_k$-representations on finite-dimensional $L$-vector spaces. Here the vector spaces are equipped with the discrete topology. For $V \in \Rep_L G_k$, we let $\mf{h}(V) \in \mc{M}(k)_L$ denote the associated Artin motive \cite[Exemples 4.1.6.1]{Andre2004}. The functor $V \mapsto \mf{h}(V)$ defines an equivalence of categories between $\Rep_L G_k$ and the full subcategory of $\mc{M}(k)_L$ generated by Artin motives. The category of motives of abelian type over $k$ with coefficients in $L$ is defined to be the thick rigid tensor subcategory of $\mc{M}(k)_L$ spanned by the Artin motives together with the motives of abelian varieties over $k$. Motives of abelian type are finite-dimensional \cite[Example 9.1]{Kimura2005}.

    \subsection{Realizations}
    The letter $\ell$ will denote a fixed prime invertible in $k$. For a smooth projective variety $X$ over $k$ and an algebraic extension $F / \Q_\ell$, we let $H^\bullet_{\et}(X_{\ol{k}}, F)$ denote the $\ell$-adic cohomology of $X$. When $k = \Co$ and $F'$ is a field, we let $H^\bullet_{B}(X, F')$ denote Betti cohomology of $X$. For a motive $M$, we similarly let $H^\bullet_{\et}(M, F)$ and $H^\bullet_B(M, F')$ denote the \'etale and Betti realizations of $M$.

    \subsection{Additional notation}
    For a scheme $X$ over $\F_q$, the notation $\Frob_X$ will mean the $q$-power Frobenius endomorphism of $X$. For an abelian variety $A$ over $k$, we define $\End^0(A) = \End_k(A) \otimes_\Z \Q$, where $\End_k(A)$ is the ring of endomorphisms of $A$ over $k$. For a set $S$ and integer $n \ge 0$, we let $\binom{S}{n}$ denote the set of $n$-element subsets of $S$. Subscripts on schemes (or morphisms of schemes) denote base change.

    \section{Weil motives}\label{secDecomp}
    In this section we review motivic decompositions induced by the action of a finite \'etale algebra. In particular, given an abelian variety $A$ and an embedding $K \hookrightarrow \End^0(A)$, with $K$ a quadratic number field, we construct an associated summand $\mc{W}_K(A)$ of $\mf{h}(A)$, called the Weil motive. When $k = \Co$, Weil classes live in the Betti cohomology of Weil motives. 
    \par
    Let $F$ be a field of characteristic zero, let $B$ be a finite \'etale $F$-algebra, and let $L/F$ be a field extension which splits $B$. In other words, we have a canonical isomorphism of $B$-algebras 
    \begin{gather}\label{eqFetAlgSplit}
        B \otimes_F L \cong \prod\limits_{\sigma \in \Sigma} L_\sigma,
    \end{gather}
    where $\Sigma = \Hom_{F\text{-alg}}(B, L)$, and $L_\sigma$ is a copy of $L$ equipped with $B$-algebra structure via $\sigma$. 
    \par
    Because $B$ is finite \'etale over $F$, the tensor product $F$-algebra $B \otimes_F B$ has a canonical direct factor $B_{\Delta} \cong B$. The corresponding idempotent $e_{\Delta}$ is the indicator function of the open and closed diagonal $\Spec B \hookrightarrow \Spec(B) \times_F \Spec(B)$.

    \subsection{Decomposition of $B$-module motives}
    A $B$-module object of an $F$-tensor category $C$ is an object $M$ of $C$ equipped with an $F$-algebra map $B \to \End_C(M)$. We begin by describing how a $B$-module motive naturally decomposes after extending scalars to $L$.
    \par
    Let $\mc{M}(k)_B$ denote the category of $B$-module objects in $\mc{M}(k)_F$, which we also refer to as $B$-module motives. If $M \in \mc{M}(k)_B$ is a $B$-module motive, then the idempotents splitting off factors from the product \eqref{eqFetAlgSplit} determine a canonical $B$-equivariant decomposition in $\mc{M}(k)_L$
    \begin{gather}\label{eqFToLDecomp}
        M \otimes_F L \cong \bigoplus\limits_{\sigma \in \Sigma} M_\sigma,
    \end{gather}
    where $B$ acts on $M_\sigma$ via $\sigma$. If $L/F$ is finite Galois, then the decomposition \eqref{eqFToLDecomp} is $\Gal(L / F)$-equivariant when considered as an isomorphism in $\mc{M}(k)_F$. Here $g \in \Gal(L / F)$ acts on the right-hand side by sending $M_\sigma$ to $M_{g \circ \sigma}$. Note that this action is not $L$-linear, but rather $L$-semilinear.
    \par
    If $M' \in \mc{M}(k)_B$ is another $B$-module motive, then we define the tensor product $M \otimes_B M' \in \mc{M}(k)_B$ via the projector $e_\Delta$ acting on the $(B \otimes_F B)$-module motive $M \otimes_F M'$. This equips $\mc{M}(k)_B$ with the structure of a symmetric monoidal category. We may more generally define Schur functors valued in $\mc{M}(k)_B$, such as symmetric and exterior powers, by applying the appropriate projectors to tensor powers of $B$-module motives (see \cite[Section 3.1]{Andre2004}). 
    \begin{prop}\label{propEigenMonoidal}
        For $M, M' \in \mc{M}(k)_B$ and $\sigma \in \Sigma$, there is a canonical isomorphism in $\mc{M}(k)_L$
        \begin{gather}\label{eqEigenMonoidal}
            (M \otimes_B M')_\sigma \cong M_\sigma \otimes_L M'_\sigma,
        \end{gather}
        compatible with the symmetry of the monoidal products $\otimes_B$ and $\otimes_L$. 
    \end{prop}
    \begin{proof}
        For now, assume that $L / F$ is finite. For $\sigma \in \Sigma$, write $\mf{h}(L_\sigma) \in \mc{M}(k)_F$ for the Artin motive associated with the trivial $G_k$-representation on the underlying $F$-vector space of $L_\sigma$. We equip $\mf{h}(L_\sigma)$ with a $B$-action via $\sigma$, and with the $L$-action induced by the natural identification $L_\sigma \cong L$. In particular, $\mf{h}(L_\sigma)$ is isomorphic in $\mc{M}(k)_L$ to the unit motive. Then by unwinding definitions, we have a canonical isomorphism
        \begin{gather*}
            M_\sigma \cong M \otimes_B \mf{h}(L_\sigma) \in \mc{M}(k)_L.
        \end{gather*}
        Using this, we calculate
        \begin{align}
            M_\sigma \otimes_L M'_{\sigma} &\cong (M \otimes_B \mf{h}(L_\sigma)) \otimes_L (M' \otimes_B \mf{h}(L_\sigma)) \nonumber \\
            &\cong (M \otimes_B \mf{h}(L_\sigma)) \otimes_L (\mf{h}(L_\sigma) \otimes_B M') \nonumber \\
            &\cong M \otimes_B (\mf{h}(L_\sigma) \otimes_L \mf{h}(L_\sigma)) \otimes_B M' \label{eqFLTensorAssoc} \\
            &\cong M \otimes_B \mf{h}(L_\sigma) \otimes_B M' \nonumber \\
            &\cong (M \otimes_B M') \otimes_B \mf{h}(L_\sigma) \nonumber \\
            &\cong (M \otimes_B M')_\sigma. \nonumber
        \end{align}
        To obtain \eqref{eqFLTensorAssoc} in the above computation, we use the associativity isomorphism for tensor products of bimodule objects in tensor categories (see e.g. the proof of \cite[Tag 00D2]{stacks-project}). Similar calculations imply the claimed compatibility with symmetry.
        \par
        The case where $L / F$ is infinite follows from the finite case, by choosing a finite subextension $L'/F$ of $L / F$ which splits $B$, and extending scalars to $L$. It is straightforward to check that the resulting isomorphism \eqref{eqEigenMonoidal} does not depend on this choice, by passing to the compositum in $L$ of a different such choice with $L'$.
    \end{proof}

    \begin{cor}\label{corSymEigen}
        For $M \in \mc{M}(k)_B$, $\sigma \in \Sigma$ and $n \ge 1$, there is a canonical isomorphism
        \begin{gather*}
            (\Sym^n_B M)_\sigma \cong \Sym^n_L M_\sigma
        \end{gather*}
        in $\mc{M}(k)_L$, where $\Sym^n$ denotes the symmetric power.
    \end{cor}
    \begin{proof}
        Using Proposition \ref{propEigenMonoidal}, we find that there is a canonical isomorphism
        \begin{gather*}
            (\otimes_B^n M)_\sigma \cong \otimes_L^n (M_\sigma),
        \end{gather*}
        compatible with applying the same permutation to the tensor factors on both sides. It is thus compatible with applying the projector which cuts out the symmetric power from the tensor power to both sides.
    \end{proof}

    \subsection{Decomposition of motives of abelian varieties}
    We next recall the canonical decomposition of the motive of an abelian variety. If $A$ is an abelian variety over $k$, we write $\mf{h}^n(A)$ for the $n$-th Chow--K\"unneth component of the Deninger--Murre decomposition of $\mf{h}(A) \in \mc{M}(k)_\Q$ \cite{DM1991}. The intersection product gives a canonical isomorphism 
    \begin{gather}\label{eqSymAlgH1}
        \mf{h}^\bullet(A)_F \cong \Sym^\bullet \mf{h}^1(A)_F
    \end{gather}
    of graded algebras in $\mc{M}(k)_F$ (\cite{Kunnemann1994}, \cite[Theorem 4.1]{Ancona2021}). On the other hand, if $R: \mc{M}(k)_F \to C$ is a realization functor induced by a classical Weil cohomology theory, valued in a tensor category $C$, then \eqref{eqSymAlgH1} induces a canonical isomorphism
    \begin{gather*}
        R(\mf{h}^\bullet(A)_F) \cong {\bigwedge}^\bullet R(\mf{h}^1(A)_F)
    \end{gather*}
    of graded algebras in $C$ \cite[Footnote 4]{Andre2005}. Here $\bigwedge^\bullet$ denotes the exterior algebra.

    \subsection{Weil motives}
    We now define Weil motives, which are certain submotives of abelian varieties associated with actions of quadratic number fields. 
    \par
    Set $F = \Q$. Let $K$ denote a quadratic number field, which will play the role of the finite \'etale algebra $B$ from the earlier discussion. Let $L$ be a field into which $K$ embeds, and fix a conjugate pair of embeddings $\tau, \ol{\tau}: K \hookrightarrow L$. Let $A$ be an abelian variety over $k$ of dimension $2n$, equipped with an embedding $\iota: K \hookrightarrow \End^0(A)$. 
    \par
    \begin{defn}\label{defWeilMotive}
        Define the \textbf{Weil motive} associated with the embedding $\iota$ via
        \begin{gather*}
            \mc{W}_K(A) = \Sym^{2n}_K \mf{h}^1(A) \in \mc{M}(k)_K.
        \end{gather*}
    \end{defn}
    \begin{rem}
        One may compare Definition \ref{defWeilMotive} to Agugliaro's notion of \enquote{Weil--Tate motive} \cite[Remark 2.25 and Definition 6.5]{Agugliaro2026}. The difference is that we allow the base field $k$ to be arbitrary, and we do not assume that the $\ell$-adic cohomology of $\mc{W}_K(A)$ consists of potential Tate classes. 
    \end{rem} 
    
    \begin{rem}\label{remWeilForgetToQ}
        We collect some basic observations and conventions concerning Definition \ref{defWeilMotive}.
        \begin{enumerate}[(i)]
            \item As an object of $\mc{M}(k)_K$, the Weil motive $\mc{W}_K(A)$ has rank one. We may implicitly forget the $K$-module structure and regard $\mc{W}_K(A)$ as an object of $\mc{M}(k)_\Q$, so that it has rank two (over $\Q$).
            \item The motive $\mc{W}_K(A)$ is naturally a summand of $\mf{h}^{2n}(A)$ in $\mc{M}(k)_\Q$, by regarding $\otimes_K^{2n} \mf{h}^1(A)$ as a summand of $\otimes_\Q^{2n}\mf{h}^1(A)$, and applying the symmetric power projector to both.
            \item Corollary \ref{corSymEigen} yields a canonical isomorphism in $\mc{M}(k)_L$
            \begin{gather*}
                \mc{W}_K(A) \otimes_\Q L \cong \Sym^{2n}_L \mf{h}^1(A)_\tau \oplus \Sym^{2n}_L \mf{h}^1(A)_{\ol{\tau}}.
            \end{gather*}
        \end{enumerate}
    \end{rem}

    The following lemma shows that the formation of Weil motives is compatible with products of abelian varieties equipped with $K$-action.
    \begin{lem}\label{lemWeilTensor}
        Suppose we are given another abelian variety $A'$ over $k$ with an embedding $K \hookrightarrow \End^0(A')$, such that $\dim A' = 2m$. Equip $A \times_k A'$ with the diagonal $K$-action. Then there is a canonical isomorphism
        \begin{gather*}
            \mc{W}_K(A \times_k A') \cong \mc{W}_K(A) \otimes_K \mc{W}_K(A')
        \end{gather*}
        in $\mc{M}(k)_K$.
    \end{lem}
    \begin{proof}
        Write 
        \begin{gather*}
            \mf{h}^1(A \times_k A') \cong \mf{h}^1(A) \oplus \mf{h}^1(A'),
        \end{gather*} 
        and note that $\mf{h}^1(A)$ and $\mf{h}^1(A')$ are oddly finite-dimensional \cite[Proposition 2.1]{Andre2005}, of respective ranks $2n$ and $2m$ in $\mc{M}(k)_K$. Hence we have
        \begin{align*}
            \Sym^r_K \mf{h}^1(A) &= 0,\\
            \Sym^s_K \mf{h}^1(A') &= 0
        \end{align*}
        for all $r > 2n$ and $s > 2m$ \cite[Corollaire 3.20]{Andre2005}. The binomial formula for a symmetric power of a direct sum then yields
        \begin{align*}
            \mc{W}_K(A \times_k A') &\cong \Sym^{2n+2m}_K(\mf{h}^1(A) \oplus \mf{h}^1(A')) \\
                &\cong \Sym^{2n}_K(\mf{h}^1(A)) \otimes_K \Sym^{2m}_K(\mf{h}^1(A')) \\
                &\cong \mc{W}_K(A) \otimes_K \mc{W}_K(A'). 
        \end{align*}
    \end{proof}
    \begin{rem}\label{remDecompOverDvrs}
        The constructions presented thus far in this section apply similarly to relative Chow motives of abelian schemes over a discrete valuation ring, and are compatible with base change. See \cite[Remark 4.2]{Ancona2021}.
    \end{rem}
    \par
    We now connect Weil motives to the classical notion of Weil classes. Suppose that $k \subseteq \Co$, that $K / \Q$ is imaginary quadratic, and that $L = \Co$. Let $H^1_B(A, \Co)$ denote Betti cohomology, let $H^1_B(A, \Co)_\tau$ denote the Betti realization of $\mf{h}^1(A)_{\tau}$, and similarly let $H^1_B(A, \Co)_{\ol{\tau}}$ denote the realization of $\mf{h}^1(A)_{\ol{\tau}}$. Let $a = \dim_\Co H^{1,0}_B(A, \Co)_\tau$, where
    \begin{gather*}
        H^{1,0}_B(A, \Co)_\tau = H^{1,0}_B(A, \Co) \cap H^1_B(A, \Co)_\tau.
    \end{gather*}
    By Corollary \ref{corSymEigen}, there is a decomposition
    \begin{gather*}
        \mc{W}_K(A) \otimes_\Q \Co \cong \Sym^{2n}_\Co \mf{h}^1(A)_\tau \oplus \Sym^{2n}_\Co \mf{h}^1(A)_{\ol{\tau}}
    \end{gather*}
    in $\mc{M}(k)_\Co$. The induced decomposition on Betti cohomology is
    \begin{gather}\label{eqWeilBetti}
        H^{2n}_B(\mc{W}_K(A), \Co) = \bigwedge\limits_\Co^{2n} H^1_B(A, \Co)_\tau \oplus \bigwedge\limits_\Co^{2n} H^1_B(A, \Co)_{\ol{\tau}},
    \end{gather}
    by \cite[Footnote 4]{Andre2005}. The two summands in \eqref{eqWeilBetti} have Hodge types $(a, 2n-a)$ and $(2n-a, a)$, respectively. 
    \begin{defn}[{cf. \cite{Markman2026}}]
        We say that the pair $(A, \iota: K \hookrightarrow \End^0(A))$ is an \textbf{abelian variety of Weil type} if $a = n$. In this case, the elements of $H^{2n}_B(\mc{W}_K(A), \Q)$ are called \textbf{Weil classes}. We sometimes write $(A, K)$ instead of $(A, \iota)$ if the embedding is clear from context.
    \end{defn}

    \section{CM-structures and exceptional subsets}\label{secCmStruct}
    We now specialize the constructions of Section \ref{secDecomp} to abelian varieties over finite fields. The endomorphism algebra of such a variety contains certain large commutative subalgebras, called CM-structures. The action of a CM-structure yields a very fine decomposition of the variety's motive, which lends insight into the Galois action on its $\ell$-adic realization. This section largely draws from \cite[Sections 6 and 7]{Ancona2021}.
    \par
    Let $A$ be an abelian variety of dimension $g$ over a finite field $k \cong \F_q$. The following assumption shall be in force throughout the section.

    \begin{asm}\label{asmExt}
        Replacing $k$ with a finite extension if necessary, assume the following: 
        \begin{enumerate}[(i)]
            \item that $\End^0(A) = \End^0(A_{\ol{k}})$;
            \item that for all $r \ge 0$, if $x \in H^{2r}_{\et}(A_{\ol{k}}, \Q_\ell(r))$ is a class which is fixed by a finite-index subgroup of $G_k$, then $x$ is fixed by $G_k$. 
        \end{enumerate}
    \end{asm}

    \begin{rem}
        A standard transfer argument shows that to prove the Tate conjecture for $A$, it suffices to prove the Tate conjecture for $A_{k'}$, where $k' / k$ is any finite extension. Thus in proving the conjecture for $A$, we may assume Assumption \ref{asmExt} without loss of generality. We also remark that the validity of the Tate conjecture for an abelian variety is an isogeny invariant, since an isogeny of abelian varieties induces an isomorphism between their respective Chow motives. See \cite[Lemma 5.1]{APFV2025} for more details.
        \par
        For the ease of the reader, we will try to ensure that statements in the remainder of the paper which are intended to parallel \cite{Ancona2021} closely follow Ancona's original formulations. We will thus sometimes overtly account for the possibility of needing to replace $k$ by a finite extension, even in situations where Assumption \ref{asmExt} guarantees that no such replacement is necessary.
    \end{rem}

    \begin{defn}[{\cite[Definition 6.3]{Ancona2021}}]
        When $A$ is simple, a \textbf{CM-structure} for $A$ is a choice of maximal CM field contained in $\End^0(A)$. In general, an isogeny decomposition $A \sim A_1 \times_k ... \times_k A_t$ into simple abelian varieties, and a choice of CM-structure $L_j$ for each $A_j$, determine a maximal commutative $\Q$-subalgebra $B = L_1 \times ... \times L_t \subseteq \End^0(A)$. An algebra $B$ constructed in this way is called a \textbf{CM-structure} for $A$. 
    \end{defn}
    \begin{rem}
        Any CM-structure for $A$ contains the center of $\End^0(A)$, and hence contains the Frobenius endomorphism $\Frob_A$. 
    \end{rem}
    \begin{prop}[{\cite{Tate1966}, \cite[Lemme 2]{Tate1971}}]
        There exists a CM-structure for $A$. Any CM-structure for $A$ has dimension $2g$ over $\Q$.
    \end{prop}

    The motivic decomposition induced by the action of a CM-structure will be described in terms of some related notation defined below. 

    \begin{nota}[{\cite[Notation 6.5]{Ancona2021}}]\label{notCMStruct}
        Let 
        \begin{gather*}
            B = L_1 \times ... \times L_t \subseteq \End^0(A)
        \end{gather*}
        be a CM-structure for $A$. Let $L / \Q$ be a finite Galois CM field into which each $L_j$ embeds, and fix a place $\lambda | \ell$ of $L$. Define the following notation:
        \begin{align*}
            \Sigma_j &= \Hom(L_j, L), \\
            \Sigma &= \bigsqcup\limits_j \Sigma_j \cong \Hom(B, L), \\
            G &= \Gal(L / \Q),
        \end{align*}
        where $\Hom$ denotes ring homomorphisms. Then $|\Sigma| = 2g$, and there is a natural action of $G$ on $\Sigma$ via composition. For $n \ge 0$ and $I \in \binom{\Sigma}{n}$, we let $\sigma_I$ denote the map $\prod\limits_{\sigma \in I} \sigma: B \to L$, and let $\alpha_I = \sigma_I(\Frob_A)$ denote the Frobenius eigenvalue associated with $I$. Note that $\sigma_I$ is a map of multiplicative monoids, but in general is not a map of $\Q$-algebras. We denote the conjugate of a subset $I \subseteq \Sigma$ by $\ol{I}$. 
        \par
        We may at times replace $L$ with a finite extension $L'$ which is CM and Galois over $\Q$. After such a replacement, the same notation $\Sigma_j, \Sigma, G$, etc. will be implicitly understood to refer to objects associated with $L'$ instead of $L$, e.g. $G$ will denote $\Gal(L' / \Q)$.
    \end{nota}
    For the remainder of the section, we fix a CM-structure $B$ and associated notation as in Notation \ref{notCMStruct}.

    \begin{prop}[{\cite[Propositions 6.6 and 6.7]{Ancona2021}}]\label{propCMDecomp}
        The following are true.
        \begin{enumerate}[(i)]
            \item There is a canonical decomposition in $\mc{M}(k)_L$ 
            \begin{gather}\label{eqHnDecompOverL}
                \mf{h}^n(A)_L \cong \bigoplus\limits_{I \in \binom{\Sigma}{n}} M_I.
            \end{gather}
            Here
            \begin{gather*}
                M_I = \bigotimes\limits_{\sigma \in I} M_\sigma,\\
                M_\sigma = \mf{h}^1(A)_\sigma,
            \end{gather*}
            where $\mf{h}^1(A)_\sigma$ is as in \eqref{eqFToLDecomp}. Each summand $M_I$ is of rank one.
            \item The decomposition \eqref{eqHnDecompOverL} is preserved by the action of the underlying multiplicative monoid of $B$. Here an element $b \in B$ acts on $M_I$ as $\prod\limits_{\sigma \in I} \sigma(b)$.
            \item The decomposition \eqref{eqHnDecompOverL} induces a decomposition in $\mc{M}(k)_\Q$
            \begin{gather*}
                \mf{h}^n(A) \cong \bigoplus\limits_{\mc{O} \in \binom{\Sigma}{n} / G} N_{\mc{O}},
            \end{gather*}
            where the summands are indexed by $G$-orbits of $\binom{\Sigma}{n}$, and
            \begin{gather*}
                (N_{\mc{O}})_L \cong \bigoplus\limits_{I \in \mc{O}} M_I.
            \end{gather*}
            In particular, the rank of $N_{\mc{O}}$ is equal to the size of the orbit $\mc{O}$.
        \end{enumerate}
    \end{prop}

    \begin{nota}
        For $I \in \binom{\Sigma}{n}$ and $M_I \in \mc{M}(k)_L$ a motive as in Proposition \ref{propCMDecomp}, we let $V_I = H^n_{\et}(M_I, L_\lambda)$ denote the $\lambda$-adic realization of $M_I$.
    \end{nota}
    
    We next define Lefschetz and Tate classes, and describe them in terms of the eigenvalues of Frobenius acting on the components of the decomposition of $\mf{h}(A)_L$.

    \begin{defn}\label{defTateAndLefClasses}
        For $r \ge 0$, let 
        \begin{gather*}
            \mc{D}^r(A) \subseteq \CH^r(A)_\Q
        \end{gather*}
        denote the $\Q$-subspace spanned by products of divisors. Elements of $\mc{D}^r(A)$ are called \textbf{Lefschetz classes}. We let
        \begin{gather*}
            \mc{D}^r_\ell(A) \subseteq H^{2r}_{\et}(A_{\ol{k}}, \Q_\ell(r))
        \end{gather*}
        denote the $\Q_\ell$-linear span of $\mc{D}^r(A)$ under the cycle class map. We also refer to elements of $\mc{D}^r_\ell(A)$ as Lefschetz classes when no confusion may arise. Let 
        \begin{gather*}
            \mc{T}^r_\ell(A) = H^{2r}_{\et}(A_{\ol{k}}, \Q_\ell(r))^{G_k}
        \end{gather*}
        denote the space of \textbf{Tate classes}. We let $\mc{D}^r_\lambda(A) = \mc{D}^r_\ell(A) \otimes_{\Q_\ell} L_\lambda$, and similarly define $\mc{T}^r_\lambda(A) = \mc{T}^r_\ell(A) \otimes_{\Q_\ell} L_\lambda$.
    \end{defn}
    \begin{prop}\label{propTateLefDecomp}
        For $r \ge 0$, Proposition \ref{propCMDecomp} induces canonical decompositions
        \begin{gather}
            \mc{T}^r_\lambda(A) = \bigoplus\limits_{\substack{I \in \binom{\Sigma}{2r} \\ \alpha_I = q^r}} V_I(r), \label{eqTrDecomp}\\
            \mc{D}^1_{\lambda}(A) =  \bigoplus\limits_{\substack{I \in \binom{\Sigma}{2} \\ \alpha_I = q}} V_I(1), \label{eqLef1Decomp}\\
            \mc{D}^2_{\lambda}(A) =  \bigoplus\limits_{I \in S} V_I(2), \label{eqLef2Decomp} 
        \end{gather}
        where 
        \begin{gather}
            S = \{I \in \binom{\Sigma}{4} : \exists \ P, Q \in \binom{\Sigma}{2} \text{ with } I = P \sqcup Q, \alpha_P = \alpha_Q = q\} \label{eqNonExoticSubsets}
        \end{gather}
    \end{prop}
    \begin{proof}
        For $I \in \binom{\Sigma}{2r}$, Proposition \ref{propCMDecomp} implies that $G_k$ acts trivially on $V_I(r)$ iff $\alpha_I = q^r$. This, together with the Tate conjecture for divisors on $A$ \cite{Tate1966}, yields \eqref{eqTrDecomp} and \eqref{eqLef1Decomp}. 
        \par
        Fix $P, Q \in \binom{\Sigma}{2}$ with $\alpha_P = \alpha_Q = q$. The cup product of $V_P$ and $V_Q$ is zero if $P \cap Q \neq \emptyset$, and spans $V_{P \sqcup Q}$ otherwise. This and \eqref{eqLef1Decomp} together yield \eqref{eqLef2Decomp}.
    \end{proof}
    \begin{defn}\label{defExoticSubset}
        We say that a four-element subset $I \in \binom{\Sigma}{4}$ is \textbf{Tate} if $\alpha_I = q^2$. We say that a Tate subset $I$ is \textbf{exceptional} if $I$ is not in the set $S$ defined in \eqref{eqNonExoticSubsets}. The property of a subset being exceptional is invariant under the $G$-action on $\binom{\Sigma}{4}$, and we say that a $G$-orbit $\mc{O} \in \binom{\Sigma}{4} / G$ is \textbf{exceptional} if $\mc{O} = G \cdot I$, where $I$ is exceptional. If $\mc{O}$ is an exceptional $G$-orbit, then we also say that the motive $N_{\mc{O}}$ of Proposition \ref{propCMDecomp} corresponding to $\mc{O}$ is \textbf{exceptional}. For an algebraic extension $T / \Q_\ell$, we say that a nonzero $x \in H^4_{\et}(A_{\ol{k}}, T(2))$ is \textbf{exceptional} if $x$ is in the span of all subspaces of the form $H^4_{\et}(N_{\mc{O}}, T(2))$ for some exceptional motive $N_{\mc{O}}$.
    \end{defn}
    \begin{rem}
        Proposition \ref{propTateLefDecomp} implies that $\mc{T}^2_\ell(A)$ is spanned by exceptional classes and products of divisor classes. Thus, in order to prove the Tate conjecture for an abelian fourfold $A$ over $k$, it suffices to show that the exceptional classes in $H^4_{\et}(A_{\ol{k}}, \Q_\ell(2))$ are algebraic.
    \end{rem}
    \begin{rem}\label{remExoticMotives}
        Our definitions of exceptional motives and classes are intended to parallel the notions of \enquote{exotic} motives and classes in \cite[Definition 7.1]{Ancona2021}. We say that $N_{\mc{O}}$ is exceptional when $H^4_{\et}(N_{\mc{O}}, \Q_\ell(2))$ is spanned by non-Lefschetz Tate classes. By contrast, Ancona says that $N_{\mc{O}}$ is exotic when $\Hom_{\mc{M}_{\num}(k)_\Q}(\mathbf{1}, N_{\mc{O}}(2))$ is nonzero and does not contain a nonzero Lefschetz class. Here $\mc{M}_{\num}(k)_\Q$ denotes the category of motives modulo numerical equivalence. Since assuming the existence of algebraic classes in a proof of the Tate conjecture would be circular, we have to work with Tate classes in some situations where Ancona works with algebraic classes.
        \par
        However, note that our definition of exceptional subset is equivalent to Ancona's notion of exotic subset \cite[Definition 7.8]{Ancona2021}. In the next section, we will exploit the fact that many of Ancona's arguments only depend on properties of Frobenius eigenvalues of exceptional subsets, and thus translate to our setting, without any a priori unsupported assumptions on the existence of algebraic cycles.
    \end{rem}

    \section{Reduction to Weil motives}\label{secRedWeilMotive}

    Let $A$ be an abelian fourfold over a finite field $k \cong \F_q$. We again require that Assumption \ref{asmExt} holds. The main result of this section is Corollary \ref{corExoticSubWeil}, which implies that to prove the Tate conjecture for $A$, it suffices to prove the algebraicity of Tate classes supported on certain Weil motives associated with $A$. This corollary is established in two steps: first, we show that every exceptional motive of rank two is isomorphic to a Weil motive; then we show that by changing the CM-structure, one can arrange that every exceptional motive has rank two. We again draw on several arguments from \cite[Section 7]{Ancona2021}, with some necessary modifications as indicated in Remark \ref{remExoticMotives}. 
    \par
    Throughout, we fix a CM-structure $B$ for $A$, and associated notation as in Notation \ref{notCMStruct}.
    \begin{lem}\label{lemExoticComb}
        The following are true.
        \begin{enumerate}[(i)]
            \item If $I \in \binom{\Sigma}{4}$ is an exceptional subset, then $\{I, \ol{I}\}$ forms a partition of $\Sigma$. 
            \item There are either zero, two, or four exceptional subsets.
            \item Every $G$-orbit of exceptional subsets contains either two or four elements.
        \end{enumerate}
    \end{lem}
    \begin{proof}
        Claim (i) follows from the argument of \cite[Lemma 7.10]{Ancona2021}. Likewise, claim (ii) follows from the argument of \cite[Lemma 7.13]{Ancona2021}, by substituting references to exotic classes/motives with references to exceptional classes/motives (and substituting \cite[Assumption 7.5]{Ancona2021} with our Assumption \ref{asmExt}). Claim (iii) follows from combining (i) and (ii).
    \end{proof}

    \begin{prop}\label{propTwoOrbit}
        Suppose that $\mc{O} = \{I, \ol{I}\}$ is a $G$-orbit of exceptional subsets, with associated motive $N_{\mc{O}} \in \mc{M}(k)_\Q$. There is then an imaginary quadratic field $K$ and a ring homomorphism
        \begin{gather*}
            K \hookrightarrow B \subseteq \End^0(A)
        \end{gather*}
        such that $N_{\mc{O}}$ is isomorphic to the associated Weil motive $\mc{W}_K(A)$.
    \end{prop}
    \begin{proof}
        The first part of the proof is similar to the argument of \cite[Proposition 6.12]{Agugliaro2026}. Let $H \subset G$ be the stabilizer of $I$. Then $[G:H] = 2$, and $K = L^H$ is quadratic over $\Q$. Since complex conjugation sends $I$ to $\ol{I} \neq I$, conjugation acts nontrivially on $K$, so $K$ is imaginary. We let $\tau: K \hookrightarrow L$ denote the inclusion. 
        \par
        The idea of the remainder of the proof is to show that $K$ embeds into every component CM field of $B$, so that $K$ embeds into $B$. By choosing a suitable embedding $K \hookrightarrow B$, we will arrange that every $\sigma \in I$ restricts to the inclusion $\tau: K \hookrightarrow L$. Then, by examining the decomposition of $\mf{h}^1(A)_L$ into eigenspaces for the action of $K$, we will find that $N_{\mc{O}} \cong \mc{W}_K(A)$.
        \par
        Fix a component CM field $L_j$ of $B$, as in Notation \ref{notCMStruct}. Fix also an embedding $\iota_j: L_j \hookrightarrow L$, and let $G_j = \Gal(L / \iota_j(L_j))$. The map $\iota_j$ induces a bijection from $\Sigma_j = \Hom(L_j, L)$ to $G / G_j$. Let $I_j = I \cap \Sigma_j$. Then $I_j$ contains exactly one member of each conjugate pair of maps in $\Sigma_j$, by Lemma \ref{lemExoticComb}. In particular, $I_j$ is a non-empty proper subset of $\Sigma_j$. Moreover, $I_j$ is $H$-stable, since both $I$ and $\Sigma_j$ are $H$-stable. If $G_j$ is not contained in $H$, then $H G_j = G$ and $H$ acts transitively on $G / G_j \cong \Sigma_j$, which contradicts the properties of $I_j$ observed above. Hence $G_j \subseteq H$, and 
        \begin{gather*}
            K = L^H \subseteq L^{G_j} = \iota_j(L_j).
        \end{gather*}
        \par
        The map $\iota_j$ is an isomorphism onto its image, and hence induces an embedding $\eta_j: K \hookrightarrow L_j$. We may then restrict elements of $\Sigma_j$ to $K$. Using the identifications $\Sigma_j \cong G / G_j$ and $\Hom(K,L) \cong G/H$, the restriction map $\Sigma_j \to \Hom(K, L)$ is identified with the quotient map $\rho: G / G_j \to G/H$. Since $I_j$ is $H$-stable and
        \begin{gather*}
            |I_j| = \frac{1}{2}|\Sigma_j| = [H:G_j],
        \end{gather*}
        the set $I_j$ is identified with one of the fibers of $\rho$. After possibly replacing $\eta_j$ with its conjugate, we find that $I_j$ is precisely the subset of $\sigma \in \Sigma_j$ such that $\sigma \circ \eta_j = \tau$.
        \par
        The $\eta_j$ assemble into a ring homomorphism $K \hookrightarrow B = \prod\limits_j L_j$. This induces a $K$-action on $\mf{h}^1(A)$, which in turn yields a decomposition
        \begin{gather*}
            \mf{h}^1(A)_L \cong \mf{h}^1(A)_\tau \oplus \mf{h}^1(A)_{\ol{\tau}}
        \end{gather*}
        in $\mc{M}(k)_L$, as in \eqref{eqFToLDecomp}. By the description of $I_j$ in the previous paragraph, we have 
        \begin{gather*}
            \mf{h}^1(A)_\tau \cong \bigoplus\limits_{\sigma \in I} M_\sigma, \\
            \mf{h}^1(A)_{\ol{\tau}} \cong \bigoplus\limits_{\sigma \in \ol{I}} M_\sigma,
        \end{gather*}
        where $M_\sigma$ is as in Proposition \ref{propCMDecomp}. Note that since each $M_\sigma$ is an oddly finite-dimensional object of rank one in $\mc{M}(k)_L$, we have $\Sym^n_L M_\sigma \cong 0$ for $n \ge 2$. By Remark \ref{remWeilForgetToQ}, the Weil motive $\mc{W}_K(A)$ satisfies
        \begin{align*}
            \mc{W}_K(A) \otimes_\Q L &\cong \Sym^4_L \mf{h}^1(A)_\tau \oplus \Sym^4_L \mf{h}^1(A)_{\ol{\tau}} \\
            &\cong M_I \oplus M_{\ol{I}} \\
            &\cong (N_{\mc{O}})_L
        \end{align*}
        in $\mc{M}(k)_L$. If we regard the above isomorphisms as maps in $\mc{M}(k)_\Q$, then they are $G$-equivariant. By Galois descent, this implies that $\mc{W}_K(A) \cong N_{\mc{O}}$ in $\mc{M}(k)_\Q$.
    \end{proof}

    \begin{prop}\label{propFourOrbitToTwoOrbit}
        After replacing $k$ with a finite extension, there is another CM-structure $B'$ for $A$ such that every exceptional motive $N_{\mc{O}}$ with respect to $B'$ has rank two. Equivalently, every exceptional orbit with respect to $B'$ has size two.
    \end{prop}
    \begin{proof}
        The claim follows from the proof of \cite[Proposition 7.3]{Ancona2021}, which combines loc. cit. Lemmas 7.13--7.16. The statements and proofs of the first three of these lemmas remain valid when Ancona's exotic classes/motives are substituted throughout with our exceptional classes/motives (and his Assumption 7.5 is substituted with our Assumption \ref{asmExt}). In Ancona's Lemma 7.15, one must also substitute his notation $\mc{M}_{I \cap \Sigma_X}$ with our $N_{\mc{O}}$, where $\mc{O}$ is the Galois orbit of $I \cap \Sigma_X$.
        \par
        Ancona's proof of his Lemma 7.16 is more involved. Checking that it adapts to our setting thus warrants a more detailed treatment, which is contained in the proof of the following lemma. Granting that lemma, Proposition \ref{propFourOrbitToTwoOrbit} follows immediately.
    \end{proof}

    \begin{lem}[{cf. \cite[Lemma 7.16]{Ancona2021}}]\label{lemExtAncona}
        Suppose that $A = E \times_k X$, where $E$ is a supersingular elliptic curve with $\dim_\Q \End^0(E) = 4$, and $X$ is an abelian threefold. Fix a CM-structure $B_X$ for $X$, and let $\Sigma_X, L_X, G_X$ be as in Notation \ref{notCMStruct}. 
        \par
        Suppose that there exists an orbit $\mc{O} = G_X \cdot I \in \binom{\Sigma_X}{3} / G_X$ such that the associated summand $N_{\mc{O}}$ of $\mf{h}^3(X)$ given by Proposition \ref{propCMDecomp} has rank two. Suppose also that the $G_k$-action on
        \begin{gather*}
            H^4_{\et}(N_{\mc{O}} \otimes_\Q \mf{h}^1(E), \Q_\ell(2))
        \end{gather*}
        is trivial. Then there exists a CM-structure $B_E$ for $E$, with $\Sigma_E = \{\sigma, \ol{\sigma}\}$, with the following property: after replacing $L_X$ with a finite extension $L / L_X$, where $L$ is CM and Galois over $\Q$ and splits $B_E$, we have
        \begin{gather}\label{eqSSDecomp}
            N_{\mc{O}} \otimes_\Q \mf{h}^1(E) \cong N_{\Gal(L / \Q) \cdot (I \sqcup \{\sigma\})} \oplus N_{\Gal(L / \Q) \cdot (I \sqcup \{\ol{\sigma}\})}
        \end{gather}
        in $\mc{M}(k)_\Q$. Here we equip $A$ with the product CM-structure $B = B_X \times B_E$, and $N_{\Gal(L / \Q) \cdot (I \sqcup \{\sigma\})}$ is as in Proposition \ref{propCMDecomp}. Moreover, both summands on the right-hand side of \eqref{eqSSDecomp} have rank two.
    \end{lem}
    \begin{proof}[Proof of Lemma \ref{lemExtAncona}]
        The statement of the present lemma differs from that of \cite[Lemma 7.16]{Ancona2021} only in that Ancona makes the stronger assumption that $H^4_{\et}(N_{\mc{O}} \otimes_\Q \mf{h}^1(E), \Q_\ell(2))$ is spanned by algebraic classes. He gives two proofs of his lemma, and we check that the first one still goes through under our hypotheses. We refer to Appendix \ref{appHK} for needed background on Hyodo--Kato realization.
        \par
        Since $N_{\mc{O}}$ has rank two, we may write $\mc{O} = \{I, J\}$, with $I, J \in \binom{\Sigma_X}{3}$. There exists a quadratic number field $F$ such that $(N_{\mc{O}})_F$ decomposes in $\mc{M}(k)_F$ as 
        \begin{gather*} 
            (N_{\mc{O}})_F \cong N_I \oplus N_{J},
        \end{gather*}
        where $N_I$ and $N_{J}$ have rank one. Indeed, using Proposition \ref{propCMDecomp}, we write
        \begin{gather*}
            (N_{\mc{O}})_{L_X} \cong M_I \oplus M_{J},
        \end{gather*}
        and take $F \subseteq L_X$ to be the field fixed by the stabilizer $H$ of $I$ under $G_X$. We then obtain $N_I$ from $M_I$ via $\Gal(L_X / F)$-descent, and similarly obtain $N_{J}$.
        \par
        Assume for a contradiction that $F$ embeds into $\Q_p$. Fixing such an embedding, we define a crystalline realization functor on $\mc{M}(k)_F$ by extending scalars to $\mc{M}(k)_{\Q_p}$ and then taking crystalline realization there. Let $\varphi$ denote the crystalline Frobenius, and recall from the appendix that $\Frob_A$ acts on the crystalline cohomology of $A$ via $\varphi^f$, where $q = p^f$.
        \par
        Using the comparison between Hyodo--Kato and crystalline cohomology, we observe that Ancona's only use of his stronger hypothesis in his first proof is in checking that the crystalline realizations of the exterior powers
        \begin{align*}
            P_I &= \bigwedge\limits_F^2(N_I \otimes_F \mf{h}^1(E)_F), \\
            P_{J} &= \bigwedge\limits_F^2(N_{J} \otimes_F \mf{h}^1(E)_F)
        \end{align*}
        have Newton number 4. In fact, he checks the stronger statement that $\Frob_A$ acts as multiplication by $q^4$ on these realizations. It will thus suffice to perform this same check under our assumptions. Then, following the rest of Ancona's original argument will complete the proof. 
        \par
        The fact that $\End^0(E)$ has dimension four, and in particular is a central simple $\Q$-algebra, implies that $\Frob_E$ acts on $\mf{h}^1(E)$ by a scalar $\epsilon \in \Q$. On the other hand, Proposition \ref{propCMDecomp} implies that $\Frob_X$ acts on $M_I = (N_I)_{L_X}$ by a scalar $\alpha_I \in L_X$. Because $\Frob_X$ already acts on $N_I$ before extending scalars to $L_X$, we must have $\alpha_I \in L_X^H = F$. Then $\Frob_A$ acts on $N_{I} \otimes_F \mf{h}^1(E)_F$ by $\epsilon \alpha_I$, and acts by the same value on any realization into a Weil cohomology theory. 
        \par
        If we fix a place $\lambda' | \ell$ of $F$, our assumption that $G_k$ acts trivially on $H^4_{\et}(N_{\mc{O}} \otimes_\Q \mf{h}^1(E), \Q_\ell(2))$ implies that $\Frob_A$ acts by $q^2$ on $H^4_{\et}(N_I \otimes_F \mf{h}^1(E)_F, F_{\lambda'})$. We thus find that $\epsilon \alpha_I = q^2$. It follows that $\Frob_A$ acts by $q^4$ on $P_I$, and likewise on its crystalline realization. By the same argument, $\Frob_A$ acts by $q^4$ on the crystalline realization of $P_J$. As discussed above, this allows us to conclude.
    \end{proof}

    \begin{cor}\label{corExoticSubWeil}
        After replacing $k$ with a finite extension, there exists a CM-structure $B'$ for $A$ such that every exceptional motive $N_{\mc{O}}$ associated with $B'$ is isomorphic to a Weil motive $\mc{W}_K(A)$, for some imaginary quadratic field $K \subseteq B'$.
    \end{cor}
    \begin{proof}
        Combine Propositions \ref{propTwoOrbit} and \ref{propFourOrbitToTwoOrbit}.
    \end{proof}

    \section{Lifting and specialization}\label{secLift}
    Again let $A$ be an abelian fourfold over a finite field $k \cong \F_q$, and again require that Assumption \ref{asmExt} holds. Fix a CM-structure $B$ for $A$, and associated notation as in Notation \ref{notCMStruct}. We give here two propositions which will ultimately allow us to prove Theorem \ref{thmTate} via specialization from characteristic zero.

    \begin{prop}[{\cite[Theorem 6.10 and Corollary 6.12]{Ancona2021}}]\label{propCMDecompLift}
         After replacing $k$ with a finite extension, there exists a $p$-adic field $T / \Q_p$ with residue field $k$, an abelian scheme $\mc{A}$ over $\mc{O}_T$, and an embedding $B \hookrightarrow \End^0(\mc{A})$, such that there exists a $B$-equivariant isogeny $f: \mc{A}_k \to A$. Moreover, the decompositions of $\mf{h}(A)$ and $\mf{h}(A)_L$ described in Proposition \ref{propCMDecomp} are induced via specialization by decompositions of the relative Chow motives $\mf{h}(\mc{A}) \in \mc{M}(\mc{O}_T)_\Q$ and $\mf{h}(\mc{A})_L \in \mc{M}(\mc{O}_T)_L$.
    \end{prop}

    \begin{prop}\label{propSpec}
        Let $R$ be a complete discrete valuation ring with characteristic zero fraction field $T$, and residue field $k$. Fix field isomorphisms $\Co \cong \ol{T}$ and $\Co \cong \ol{\Q}_\ell$. Let $\mc{N} \in \mc{M}(R)_\Q$ be a relative motive over $R$, with geometric generic fiber $\mc{N}_{\ol{T}}$, and geometric special fiber $\mc{N}_{\ol{k}}$. Suppose that the Betti realization of the geometric generic fiber $H^\bullet_{B}(\mc{N}_{\ol{T}}, \Co)$ is the $\Co$-span of algebraic cycles defined over $\ol{T}$. Then there exists a finite extension $k' / k$ such that the $\ell$-adic realization of the geometric special fiber $H^\bullet_{\et}(\mc{N}_{\ol{k}}, \Q_\ell)$ is the $\Q_\ell$-span of algebraic cycles defined over $k'$. 
    \end{prop}
    \begin{proof}
        The Artin comparison theorem between Betti cohomology with coefficients in $\Co$ and \'etale cohomology with coefficients in $\ol{\Q}_\ell$ is compatible with pullbacks, cycle class maps, Tate twists, Poincar\'e duality, and cup products (see e.g. \cite{SGA4III}, in particular Th\'eor\`eme XVI.4.1, and \cite[Section 1]{Deligne1982}). It is hence compatible with the composition of algebraic correspondences, and induces a natural tensor isomorphism of realization functors on $\mc{M}(\Co)_\Q$ valued in graded $\Co$-vector spaces. The fact that $H^\bullet_{B}(\mc{N}_{\ol{T}}, \Co)$ is spanned by classes of algebraic cycles defined over $\ol{T}$ thus implies that $H^\bullet_{\et}(\mc{N}_{\ol{T}}, \ol{\Q}_\ell)$ is spanned by the classes of algebraic cycles defined over $\ol{T}$. We may then choose a finite spanning set of cycles which are all defined over a finite extension $T' / T$, with residue field extension $k' / k$. We deduce via specialization (\cite[Theorem 6.25]{Kahn2020}, \cite[Example 20.3.5]{Fulton1998}) that $H^\bullet_{\et}(\mc{N}_{\ol{k}}, \ol{\Q}_\ell)$ is spanned by classes of algebraic cycles defined over $k'$. Since the extension of scalars functor from $\Q_\ell$-vector spaces to $\ol{\Q}_\ell$-vector spaces is faithful and exact, we obtain the same statement with $\Q_\ell$-coefficients.
    \end{proof}

    \section{Proof of Theorem \ref{thmTate}}\label{secProofThmTate}
    Again let $A$ be an abelian fourfold over a finite field $k \cong \F_q$, and again require that Assumption \ref{asmExt} holds. Fix a CM-structure $B$ for $A$ as in Corollary \ref{corExoticSubWeil}, and associated notation as in Notation \ref{notCMStruct}. The corollary implies that to prove the Tate conjecture for $A$, it suffices to prove that for all imaginary quadratic subfields $K \subseteq B$ such that the associated Weil motive $\mc{W}_K(A)$ is exceptional (Definition \ref{defExoticSubset}), the $\ell$-adic realization $H^4_{\et}(\mc{W}_K(A), \Q_\ell(2))$ is spanned by algebraic classes. We now fix such a $K$.
    \par
    Replacing $k$ with a finite extension if necessary, we fix a finite extension $T / \Q_p$ with residue field $k$, and an abelian scheme $\mc{A}$ over $\mc{O}_T$, as in Proposition \ref{propCMDecompLift}. The proposition gives a direct summand $\mc{W}_K(\mc{A})$ of $\mf{h}(\mc{A})$ which specializes to $\mc{W}_K(A)$.
    \par
    Fix also isomorphisms $\ol{\Q}_p \cong \ol{T}$ and $\ol{T} \cong \Co$, and an embedding $\tau: K \hookrightarrow \Co$. We have inclusions $K \subseteq B \subseteq \End^0(\mc{A})$, and we let
    \begin{gather*}
        a = \dim_\Co H^{1,0}_B(\mc{A}_\Co, \Co)_\tau
    \end{gather*}
    be the dimension of the $\tau$-eigenspace of $K$ acting on $H^{1,0}_B(\mc{A}_\Co, \Co)$. Then \eqref{eqWeilBetti} exhibits the Betti realization $H^4_B(\mc{W}_K(\mc{A}_{\Co}), \Co)$ as the direct sum of two one-dimensional Hodge structures, with respective Hodge types $(a, 4-a)$ and $(4-a, a)$. Replacing $\tau$ with its conjugate if necessary, we may assume that $a \in \{2, 3, 4\}$.
    \par
    Suppose first that $a = 2$. Then the Betti realization of $\mc{W}_K(\mc{A}_{\Co})$ is spanned by Hodge classes, so the proof of the Hodge conjecture for abelian fourfolds \cite[Corollary 1.6.1]{Markman2025} implies that $H^4_B(\mc{W}_K(\mc{A}_{\Co}), \Co(2))$ is spanned by algebraic classes. After replacing $k$ with a finite extension if necessary, we find from Proposition \ref{propSpec} that $H^4_{\et}(\mc{W}_K(A), \Q_\ell(2))$ is also spanned by algebraic classes.
    \par
    In the case $a = 4$, the argument of \cite[Remark A.8(2)]{Ancona2021} implies that $A$ is isogenous to a power of a supersingular elliptic curve, and hence satisfies the Tate conjecture (e.g. by \cite[Theorem 2.9]{FL2021}). We thus assume that $a = 3$ for the remainder of the section.
    \begin{lem}\label{lemPNonSplit}
        The prime $p$ does not split in $K$.
    \end{lem}
    \begin{proof}
        We again follow the argument of \cite[Lemma 7.16]{Ancona2021}, referring to Appendix \ref{appHK} for background on Hyodo--Kato realization. Suppose for a contradiction that $K$ embeds in $\Q_p$, and let $\rho, \ol{\rho}: K \hookrightarrow \Q_p$ be a conjugate pair of embeddings. As in \eqref{eqFToLDecomp}, we obtain a decomposition
        \begin{gather*}
            \mc{W}_K(\mc{A}) \otimes_\Q \Q_p \cong \mc{M}_{\rho} \oplus \mc{M}_{\ol{\rho}}
        \end{gather*}
        in $\mc{M}(\mc{O}_T)_{\Q_p}$. Let $H_\rho$ and $H_{\ol{\rho}}$ be the respective Hyodo--Kato realizations of the summands. 
        \par
        Using our previously fixed isomorphism $\ol{\Q}_p \cong \Co$, the Betti realizations of $\mc{M}_\rho$ and $\mc{M}_{\ol{\rho}}$ are one-dimensional, and we may assume that they have respective Hodge types $(3, 1)$ and $(1, 3)$. By compatibility of Hyodo--Kato and de Rham realizations, we compute the Hodge numbers as
        \begin{gather*}
            t_H(H_\rho) = 3, \\
            t_H(H_{\ol{\rho}}) = 1.
        \end{gather*}
        Note that because $\mc{W}_K(A)$ is exceptional in the sense of Definition \ref{defExoticSubset}, the Frobenius $\Frob_A$ acts on $\mc{W}_K(A) \otimes_\Q L$ via $q^2$. Since the extension of scalars functor $\mc{M}(k)_\Q \to \mc{M}(k)_L$ is faithful, this implies that $\Frob_A$ acts on $\mc{W}_K(A)$ itself via $q^2$. Then $\Frob_A$ also acts on $H_\rho$ and $H_{\ol{\rho}}$ by $q^2$. But $\Frob_A$ acts on the $\varphi$-modules $H_\rho$ and $H_{\ol{\rho}}$ via $\varphi^f$, so the Newton numbers are
        \begin{gather*} 
            t_N(H_\rho) = t_N(H_{\ol{\rho}}) = 2.
        \end{gather*}
        This contradicts the weak admissibility of $H_\rho$ and $H_{\ol{\rho}}$.
    \end{proof}
    We will now make use of some standard facts and terminology from the theory of complex multiplication of abelian varieties, referring to \cite[Chapter 10]{Milne2017} for background. Replacing $T$ with a finite extension if necessary, there exists an elliptic curve $E$ over $T$ with an isomorphism $\eta: K \cong \End^0(E)$, such that $E$ has good reduction. Let $\mc{E}$ be the N\'eron model of $E$ over $\mc{O}_T$. By Lemma \ref{lemPNonSplit} and results of Deuring \cite[Theorem 1.1]{Blake2014}, the reduction $\mc{E}_k$ is supersingular, and we may further assume that $\dim_\Q \End^0(\mc{E}_k) = 4$. Choose $\eta$ such that the CM-type of $(E_\Co, \eta)$ is the conjugate of $\tau$. 
    \par
    Form the product $\mc{Y} = \mc{A} \times_{\mc{O}_T} \mc{E}^2$, and equip it with the diagonal $K$-action. Then the associated Weil motive satisfies 
    \begin{gather}\label{eqWeilYTensor}
        \mc{W}_K(\mc{Y}) \cong \mc{W}_K(\mc{A}) \otimes_K \mc{W}_K(\mc{E}^2),
    \end{gather}
    by Lemma \ref{lemWeilTensor} and Remark \ref{remDecompOverDvrs}. After extending the coefficient field from $\Q$ to $\Co$, Proposition \ref{propEigenMonoidal} and Corollary \ref{corSymEigen} yield an identification
    \begin{gather}\label{eqYSymDecomp}
        \mc{W}_K(\mc{Y}_\Co)_\tau \cong \Sym^4_\Co \mf{h}^1(\mc{A}_\Co)_\tau \otimes_\Co \Sym^2_\Co \mf{h}^1(E^2_\Co)_\tau,
    \end{gather}
    and similarly when $\tau$ is replaced with $\ol{\tau}$. Since
    \begin{gather*}
        H^1_B(E_\Co, \Co)_\tau = H^{0, 1}(E_\Co, \Co)
    \end{gather*}
    and $a = 3$, the formula \eqref{eqYSymDecomp} implies that $H^6_B(\mc{W}_K(\mc{Y}_\Co), \Co)$ is of pure Hodge type $(3, 3)$. Hence $(\mc{Y}_\Co, K)$ is an abelian sixfold of Weil type. 
    \par
    By \cite[Proposition 1.6]{Milne2022}, there exists a polarization $\lambda$ of $\mc{Y}_\Co = (\mc{A}_\Co \times_{\Co} E_\Co) \times_\Co E_\Co$ such that $\lambda$ has determinant $-u$, where $u$ lies in the image of the field norm $K^\times \to \Q^\times$. Now \cite[Theorem 1.5.1]{Markman2025} implies that $H^6_B(\mc{W}_K(\mc{Y}_{\Co}), \Co(3))$ is spanned by algebraic classes. From Proposition \ref{propSpec} we then find that $H^6_{\et}(\mc{W}_K(\mc{Y}_k), \Q_\ell(3))$ is spanned by algebraic classes.
    \par
    It remains to show that $H^4_{\et}(\mc{W}_K(A), \Q_\ell(2))$ is spanned by algebraic classes. Because $\mc{E}_k$ is supersingular with $\dim_\Q \End^0(\mc{E}_k) = 4$, we have
    \begin{gather}\label{eqWeilMotiveIsTate}
        \mc{W}_K(\mc{E}^2_k) \cong \mathbf{1}(-1)^{\oplus 2}
    \end{gather}
    in $\mc{M}(k)_\Q$, by the proof of \cite[Theorem 2.9]{FL2021}. Since all one-dimensional $K$-vector spaces are isomorphic, any two $K$-module structures on $\mathbf{1}(-1)^{\oplus 2}$ are isomorphic in $\mc{M}(k)_K$. The isomorphism \eqref{eqWeilMotiveIsTate} thus implies the existence of an isomorphism
    \begin{gather*}
        \mc{W}_K(\mc{E}^2_k) \cong \mathbf{1}_K(-1)
    \end{gather*}
    in $\mc{M}(k)_K$. Now \eqref{eqWeilYTensor} implies that
    \begin{gather*}
        \mc{W}_K(A)(2) \cong \mc{W}_K(\mc{Y}_k)(3).
    \end{gather*}
    The $\ell$-adic realization of $\mc{W}_K(\mc{Y}_k)(3)$ is spanned by algebraic classes, so the same is true for the realization of $\mc{W}_K(A)(2)$. This completes the proof of Theorem \ref{thmTate}.

    \section{Consequences of Theorem \ref{thmTate}}\label{secStd}
    In this section we present some corollaries of Theorem \ref{thmTate}. In particular, we prove Grothendieck's standard conjecture $D$ for abelian fourfolds. 
    \par
    Denote rational, $\ell$-adic homological, and numerical equivalence by the respective shorthands $\sim_{\rat}, \sim_\ell,$ and $\sim_{\num}$. For a smooth projective variety $X$ and adequate equivalence relations $\sim$ and $\sim'$, we will write \enquote{$\sim \ = \ \sim' \text{ on } X$} to mean that $\sim$ and $\sim'$ agree for algebraic cycles with rational coefficients on $X$.
    \par
    Again let $k \cong \F_q$ be a finite field. Let $X$ be a smooth projective variety over $k$, and assume that $X$ has motive of abelian type. Kahn \cite{Kahn2003} showed that the Tate conjecture for such an $X$ has many powerful motivic consequences. By Theorem \ref{thmTate}, these consequences hold unconditionally when $X$ is an abelian fourfold. We collect a few of them below.
    \begin{cor}\label{corMotConsequences}
        Let $A$ be an abelian fourfold over $k$.
        \begin{enumerate}[(i)]
            \item We have $\sim_{\rat} \ = \ \sim_{\num}$ on $A$ \cite[Th\'eor\`eme 1]{Kahn2003}.
            \item The crystalline Tate conjecture holds for $A$, by (i) and \cite[Theorem 1.2]{Milne2019}.
            \item The Weil-\'etale motivic cohomology groups $H^i_W(A, \Z(n))$ of $A$ are finitely generated, and the special values of the zeta function of $A$ satisfy the formula conjectured by Lichtenbaum \cite[Corollaires 3.8 and 3.10]{Kahn2003}.
            \item The Chow group with integer coefficients $\CH^2(A)$ is finitely generated \cite[Th\'eor\`eme 4.5]{Kahn2003}.
        \end{enumerate}
    \end{cor}
    \begin{cor}\label{corStd}
        Let $A'$ be an abelian fourfold over a field $k'$ in which $\ell$ is invertible. Then $\sim_\ell \ = \ \sim_{\num}$ on $A'$.
    \end{cor}
    \begin{proof}
        By spreading out, we may assume without loss of generality that $k'$ is a finitely generated field. There then exists an integral finite-type scheme $S$ over $\Spec \Z[1/\ell]$ such that $S$ has fraction field $k'$, and an abelian scheme $\mc{A}'$ over $S$ with generic fiber $A'$. If $s \in S$ is a closed point, then Corollary \ref{corMotConsequences}(i) implies that $\sim_\ell \ = \ \sim_{\num}$ on the fiber $\mc{A}'_s$ over $s$. Now apply \cite[Corollary 3.19]{Ancona2021}.
    \end{proof}
    Lieberman proved the standard conjectures of Lefschetz and K\"unneth type for abelian varieties \cite[Theorem 2A11]{Kleiman1968}, and Ancona proved the standard conjecture of Hodge type for abelian fourfolds \cite[Theorem 3.18]{Ancona2021}. Corollary \ref{corStd} asserts that the one remaining standard conjecture, standard conjecture $D$, also holds for abelian fourfolds. 

    \appendix 
    \section{Review of Hyodo--Kato realization}\label{appHK}
    We briefly recall the properties of the Hyodo--Kato realization we use in the main text, following \cite[Convention (4)]{Ancona2021}. 
    \par
    Let $T$ be a finite extension of $\Q_p$, with ring of integers $\mc{O}_T$ and residue field $k \cong \F_q$. Let $T_0 = W(k)[1/p]$ be the maximal unramified subextension of $T / \Q_p$, and let $\sigma: T_0 \to T_0$ denote the Witt vector Frobenius map.
    \par
    A \textbf{filtered} $\varphi$\textbf{-module} $\mc{D} = (D, \varphi, \Fil^\bullet D_T)$ over $T$ is defined to consist of the following data \cite[Definition 7.3.4]{BC2009}:
    \begin{enumerate}[(i)]
        \item a finite-dimensional $T_0$-vector space $D$;
        \item an additive bijection $\varphi: D \to D$ which is Frobenius-semilinear, meaning that for all $\lambda \in T_0$ and $x \in D$, we have $\varphi(\lambda \cdot x) = \sigma(\lambda) \varphi(x)$;
        \item a decreasing filtration $\Fil^\bullet D_T$ on $D_T = D \otimes_{T_0} T$, such that $\Fil^i D_T = D_T$ for $i \ll 0$ and $\Fil^i D_T = 0$ for $i \gg 0$.
    \end{enumerate}
    Define the \textbf{Hodge number}
    \begin{gather*}
        t_H(\mc{D}) = \sum\limits_{i \in \Z} i \dim_T \gr^i D_T,
    \end{gather*}
    cf. \cite[Definition 8.1.1]{BC2009}. Note that $\varphi^f: D \to D$ is a $T_0$-linear map, so we may take its determinant. Define the \textbf{Newton number} 
    \begin{gather*}
        t_N(\mc{D}) = \frac{1}{f} v_p(\det\limits_{T_0} (\varphi^f |D)),
    \end{gather*}
    where $v_p$ denotes the $p$-adic valuation with $v_p(p) = 1$ \cite[Definition 8.1.7 and Proposition 8.1.9]{BC2009}.
    \par
    Let $D' \subseteq D$ be a $\varphi$-stable $T_0$-subspace, and equip $D'_T$ with the induced filtration $\Fil^i D'_T = D'_T \cap \Fil^i D_T$. Then $\mc{D}' = (D', \varphi|_{D'}, \Fil^\bullet D'_T)$ is itself a filtered $\varphi$-module. We say that $\mc{D}$ is \textbf{weakly admissible} if $t_H(\mc{D}) = t_N(\mc{D})$, and for all $\mc{D}'$ constructed as above, we have $t_H(\mc{D}') \le t_N(\mc{D}')$ \cite[Definition 8.2.1]{BC2009}. We let $\MOD_T$ denote the category of weakly admissible filtered $\varphi$-modules over $T$. It follows from a theorem of Colmez--Fontaine \cite[Th\'eor\`eme A]{CF2000} that weak admissibility is equivalent to the a priori stronger condition of admissibility.
    \par
    There is a \textbf{Hyodo--Kato} realization functor
    \begin{gather*}
        R_{\HK}: \mc{M}(\mc{O}_T)_\Q \to \MOD_T
    \end{gather*}
    as in \cite[Convention (4)]{Ancona2021} (see also \cite[Remark 4.12 and Section 4.15]{DN2018}). Let $\mc{N} \in\mc{M}(\mc{O}_T)_\Q$ be a relative motive and $R_{\HK}(\mc{N}) = (D, \varphi, \Fil^\bullet D_T)$ its Hyodo--Kato realization. Then the filtered $T$-vector space $D_T$ is canonically identified with the de Rham realization $R_{\dR}(\mc{N}_T)$ of the generic fiber of $\mc{N}$. The $\varphi$-module $(D, \varphi)$ is canonically identified with the crystalline realization $R_{\cris}(\mc{N}_k)$ of the special fiber of $\mc{N}$. Since the $\varphi$-action on the crystalline cohomology of a smooth projective variety over $k$ is induced by pullback along the absolute $p$-power Frobenius map \cite[Tag 07N0]{stacks-project}, the $q$-power Frobenius endomorphism of $\mc{N}_k$ acts on $R_{\cris}(\mc{N}_k)$ as $\varphi^f$.

    \makeatletter
    {
        \let\@tocwrite\@gobbletwo
        \section*{Acknowledgements}
    }
    \makeatother
    The author is grateful to Giuseppe Ancona for thoughtful comments. 
    \bigskip
    
    \makeatletter
    {
        \let\@tocwrite\@gobbletwo
        \section*{Tool and computational resource disclosure}
    }
    \makeatother
    A large language model\footnote{The author has opted not to identify the specific model used here, for reasons similar to those given in \cite[Footnote 3]{Kelly2026}.} was used in the preparation of this manuscript to explore proof strategies, to locate relevant literature, to critically examine mathematical arguments, and to assist with copyediting and LaTeX typesetting. The author initially asked the model to attempt to use Markman's results to prove Theorem \ref{thmTate}, and proposed technical directions that a solution could take, including translating Markman's program directly to positive characteristic, or applying Milne's formalism of Lefschetz groups \cite{Milne1999}. In the ensuing interaction, the model suggested a method involving Ancona's analysis of exotic classes \cite{Ancona2021} and the \enquote{stabilization} idea discussed in Remark \ref{remStabilize}. The author subsequently developed this strategy, significantly restructured it, identified and corrected a number of errors and gaps, and introduced several simplifications and reformulations. In particular, the author systematically based the proof on the theory of Chow motives, and devised a motivic technique which simplified the last step of the preliminary approach.
    \par
    The text of the manuscript was composed by the author. All mathematical statements and cited references in the manuscript were independently checked by the author. Sole responsibility for the content of the manuscript lies with the author.
    \bigskip

    \printbibliography

@misc{APFV2025,
  author       = {Arango-Pi{\~n}eros, Santiago and Frengley, Sam and Vemulapalli, Sameera},
  title        = {Galois groups of simple abelian varieties over finite fields and exceptional {T}ate classes},
  year         = {2025},
  eprint       = {2505.09589},
  eprinttype   = {arxiv},
  eprintclass  = {math.NT},
  version      = {v1},
  url          = {https://arxiv.org/abs/2505.09589}
}

@article{Ancona2021,
  author       = {Ancona, Giuseppe},
  title        = {Standard conjectures for abelian fourfolds},
  journal      = {Inventiones Mathematicae},
  volume       = {223},
  number       = {1},
  pages        = {149--212},
  year         = {2021},
  doi          = {10.1007/s00222-020-00990-7},
  url          = {https://doi.org/10.1007/s00222-020-00990-7}
}

@book{Andre2004,
  author       = {Andr{\'e}, Yves},
  title        = {Une introduction aux motifs (motifs purs, motifs mixtes, p{\'e}riodes)},
  series       = {Panoramas et Synth{\`e}ses},
  volume       = {17},
  publisher    = {Soci{\'e}t{\'e} Math{\'e}matique de France},
  location     = {Paris},
  year         = {2004},
  isbn         = {2-85629-164-3}
}

@incollection{Andre2005,
  author       = {Andr{\'e}, Yves},
  title        = {Motifs de dimension finie (d'apr{\`e}s {S.-I. Kimura, P. O'Sullivan,} etc.)},
  booktitle    = {S{\'e}minaire Bourbaki, volume 2003/2004, expos{\'e}s 924--937},
  editor       = {{Association des amis de Nicolas Bourbaki}},
  series       = {Ast{\'e}risque},
  number       = {299},
  pages        = {115--145},
  publisher    = {Soci{\'e}t{\'e} Math{\'e}matique de France},
  location     = {Paris},
  year         = {2005},
  note         = {Expos{\'e} no. 929},
  doi          = {10.24033/ast.679},
  url          = {https://doi.org/10.24033/ast.679}
}

@online{BC2009,
  author       = {Brinon, Olivier and Conrad, Brian},
  title        = {{CMI} Summer School Notes on {$p$}-adic Hodge Theory},
  subtitle     = {Preliminary Version},
  year         = {2009},
  url          = {https://math.stanford.edu/~conrad/papers/notes.pdf},
  urldate      = {2026-08-06}
}

@article{Blake2014,
  author       = {Blake, Chris},
  title        = {A Deuring criterion for abelian varieties},
  journal      = {Bulletin of the London Mathematical Society},
  volume       = {46},
  number       = {6},
  pages        = {1256--1263},
  year         = {2014},
  doi          = {10.1112/blms/bdu079},
  url          = {https://doi.org/10.1112/blms/bdu079}
}

@article{CF2000,
  author       = {Colmez, Pierre and Fontaine, Jean-Marc},
  title        = {Construction des repr{\'e}sentations {$p$}-adiques semi-stables},
  journal      = {Inventiones Mathematicae},
  volume       = {140},
  number       = {1},
  pages        = {1--43},
  year         = {2000},
  doi          = {10.1007/s002220000042},
  url          = {https://doi.org/10.1007/s002220000042}
}

@article{DM1991,
  author       = {Deninger, Christopher and Murre, Jacob},
  title        = {Motivic decomposition of abelian schemes and the {F}ourier transform},
  journal      = {Journal f{\"u}r die reine und angewandte Mathematik},
  volume       = {422},
  pages        = {201--219},
  year         = {1991},
  doi          = {10.1515/crll.1991.422.201},
  url          = {https://doi.org/10.1515/crll.1991.422.201}
}

@article{DN2018,
  author       = {D{\'e}glise, Fr{\'e}d{\'e}ric and Nizio{\l}, Wies{\l}awa},
  title        = {On {$p$}-adic absolute {H}odge cohomology and syntomic coefficients. {I}},
  journal      = {Commentarii Mathematici Helvetici},
  volume       = {93},
  number       = {1},
  pages        = {71--131},
  year         = {2018},
  doi          = {10.4171/CMH/430},
  url          = {https://doi.org/10.4171/CMH/430}
}

@collection{SGA4III,
  editor       = {Artin, Michael and Grothendieck, Alexander and Verdier, Jean-Louis},
  title        = {Th{\'e}orie des topos et cohomologie {\'e}tale des sch{\'e}mas},
  subtitle     = {Tome 3},
  series       = {Lecture Notes in Mathematics},
  volume       = {305},
  publisher    = {Springer-Verlag},
  location     = {Berlin},
  year         = {1973},
  note         = {S{\'e}minaire de G{\'e}om{\'e}trie Alg{\'e}brique du Bois-Marie 1963--1964 (SGA 4), avec la collaboration de P. Deligne et B. Saint-Donat},
  doi          = {10.1007/BFb0070714},
  url          = {https://doi.org/10.1007/BFb0070714}
}

@inbook{Deligne1982,
  author       = {Deligne, Pierre},
  title        = {Hodge cycles on abelian varieties},
  bookauthor   = {Deligne, Pierre and Milne, James S. and Ogus, Arthur and Shih, Kuang-yen},
  booktitle    = {Hodge Cycles, Motives, and Shimura Varieties},
  series       = {Lecture Notes in Mathematics},
  volume       = {900},
  pages        = {9--100},
  publisher    = {Springer-Verlag},
  location     = {Berlin and New York},
  year         = {1982},
  note         = {Notes by James S. Milne; corrected TeXed version with endnotes, revised 2018},
  doi          = {10.1007/978-3-540-38955-2},
  url          = {https://www.jmilne.org/math/Documents/Deligne82.pdf},
  urldate      = {2026-08-06}
}

@incollection{FL2021,
  author       = {Fu, Lie and Li, Zhiyuan},
  title        = {Supersingular irreducible symplectic varieties},
  booktitle    = {Rationality of Varieties},
  editor       = {Farkas, Gavril and van der Geer, Gerard and Shen, Mingmin and Taelman, Lenny},
  series       = {Progress in Mathematics},
  volume       = {342},
  pages        = {147--200},
  publisher    = {Birkh{\"a}user/Springer},
  location     = {Cham},
  year         = {2021},
  doi          = {10.1007/978-3-030-75421-1_7},
  url          = {https://doi.org/10.1007/978-3-030-75421-1_7}
}

@book{Fulton1998,
  author       = {Fulton, William},
  title        = {Intersection Theory},
  edition      = {2},
  series       = {Ergebnisse der Mathematik und ihrer Grenzgebiete. 3. Folge},
  volume       = {2},
  publisher    = {Springer-Verlag},
  location     = {Berlin},
  year         = {1998},
  doi          = {10.1007/978-1-4612-1700-8},
  url          = {https://doi.org/10.1007/978-1-4612-1700-8}
}

@misc{Jiang2025,
  author       = {Jiang, Ruofan},
  title        = {{$p$}-adic monodromy and mod {$p$} unlikely intersections, {II}},
  year         = {2025},
  eprint       = {2512.00687},
  eprinttype   = {arxiv},
  eprintclass  = {math.NT},
  version      = {v1},
  url          = {https://arxiv.org/abs/2512.00687}
}

@book{Kahn2020,
  author       = {Kahn, Bruno},
  title        = {Zeta and {$L$}-Functions of Varieties and Motives},
  series       = {London Mathematical Society Lecture Note Series},
  volume       = {462},
  publisher    = {Cambridge University Press},
  location     = {Cambridge},
  year         = {2020},
  doi          = {10.1017/9781108691536},
  url          = {https://doi.org/10.1017/9781108691536}
}

@article{Kimura2005,
  author       = {Kimura, Shun-Ichi},
  title        = {Chow groups are finite dimensional, in some sense},
  journal      = {Mathematische Annalen},
  volume       = {331},
  number       = {1},
  pages        = {173--201},
  year         = {2005},
  doi          = {10.1007/s00208-004-0577-3},
  url          = {https://doi.org/10.1007/s00208-004-0577-3}
}

@incollection{Kunnemann1994,
  author       = {K{\"u}nnemann, Klaus},
  title        = {On the {C}how motive of an abelian scheme},
  booktitle    = {Motives},
  editor       = {Jannsen, Uwe and Kleiman, Steven L. and Serre, Jean-Pierre},
  series       = {Proceedings of Symposia in Pure Mathematics},
  volume       = {55},
  note         = {Part 1},
  pages        = {189--205},
  publisher    = {American Mathematical Society},
  location     = {Providence, RI},
  year         = {1994},
  doi          = {10.1090/pspum/055.1/1265530},
  url          = {https://doi.org/10.1090/pspum/055.1/1265530}
}

@article{MZ1999,
  author       = {Moonen, Ben J. J. and Zarhin, Yuri G.},
  title        = {Hodge classes on abelian varieties of low dimension},
  journal      = {Mathematische Annalen},
  volume       = {315},
  number       = {4},
  pages        = {711--733},
  year         = {1999},
  doi          = {10.1007/s002080050333},
  url          = {https://doi.org/10.1007/s002080050333}
}

@misc{Markman2025,
  author       = {Markman, Eyal},
  title        = {Cycles on abelian {$2n$}-folds of {W}eil type from secant sheaves on abelian {$n$}-folds},
  year         = {2025},
  eprint       = {2502.03415},
  eprinttype   = {arxiv},
  eprintclass  = {math.AG},
  version      = {v2},
  url          = {https://arxiv.org/abs/2502.03415}
}

@article{Milne1999b,
  author       = {Milne, James S.},
  title        = {Lefschetz motives and the {T}ate conjecture},
  journal      = {Compositio Mathematica},
  volume       = {117},
  number       = {1},
  pages        = {47--81},
  year         = {1999},
  doi          = {10.1023/A:1000776613765},
  url          = {https://www.jmilne.org/math/articles/1999b.html},
  addendum     = {See the correction and corrected version on the author's website}
}

@incollection {Tate1971,
    AUTHOR = {Tate, John},
     TITLE = {Classes d'isog\'enie des vari\'et\'es ab\'eliennes sur un
              corps fini (d'apr\`es {T}. {H}onda)},
 BOOKTITLE = {S\'eminaire {B}ourbaki. {V}ol. 1968/69: {E}xpos\'es 347--363},
    SERIES = {Lecture Notes in Math.},
    VOLUME = {175},
     PAGES = {Exp. No. 352, 95--110},
 PUBLISHER = {Springer, Berlin},
      YEAR = {1971},
      ISBN = {3-540-05356-5; 0-387-05356-6},
   MRCLASS = {14K02},
  MRNUMBER = {3077121},
}

@online{Milne2017,
  author       = {Milne, James S.},
  title        = {Introduction to Shimura Varieties},
  year         = {2017},
  note         = {Revised version},
  url          = {https://www.jmilne.org/math/xnotes/svi.pdf},
  urldate      = {2026-08-06}
}

@misc{Milne2022,
  author       = {Milne, James S.},
  title        = {The {T}ate and standard conjectures for certain abelian varieties},
  year         = {2022},
  eprint       = {2112.12815},
  eprinttype   = {arxiv},
  eprintclass  = {math.NT},
  version      = {v2},
  url          = {https://arxiv.org/abs/2112.12815}
}

@article{OSullivan2011,
  author       = {O'Sullivan, Peter},
  title        = {Algebraic cycles on an abelian variety},
  journal      = {Journal f{\"u}r die reine und angewandte Mathematik},
  volume       = {654},
  pages        = {1--81},
  year         = {2011},
  doi          = {10.1515/CRELLE.2011.025},
  url          = {https://doi.org/10.1515/CRELLE.2011.025}
}

@article{Pohlmann1968,
  author       = {Pohlmann, Henry},
  title        = {Algebraic cycles on abelian varieties of complex multiplication type},
  journal      = {Annals of Mathematics},
  series       = {2},
  volume       = {88},
  number       = {2},
  pages        = {161--180},
  year         = {1968},
  doi          = {10.2307/1970570},
  url          = {https://doi.org/10.2307/1970570}
}

@incollection{Tate1965,
  author       = {Tate, John},
  title        = {Algebraic cycles and poles of zeta functions},
  booktitle    = {Arithmetical Algebraic Geometry},
  editor       = {Schilling, O. F. G.},
  note         = {Proceedings of the Conference at Purdue University, 1963},
  pages        = {93--110},
  publisher    = {Harper \& Row},
  location     = {New York},
  year         = {1965}
}

@article{Tate1966,
  author       = {Tate, John},
  title        = {Endomorphisms of abelian varieties over finite fields},
  journal      = {Inventiones Mathematicae},
  volume       = {2},
  number       = {2},
  pages        = {134--144},
  year         = {1966},
  doi          = {10.1007/BF01404549},
  url          = {https://doi.org/10.1007/BF01404549}
}

@article{Schoen1998, 
    title={Addendum to: Hodge Classes on Self-Products of a Variety with an Automorphism}, 
    volume={114}, 
    doi={10.1023/A:1000566205021}, 
    number={3}, 
    journal={Compositio Mathematica}, 
    author={Schoen, Chad}, 
    year={1998}, 
    pages={321–328}
}

@article {Tankeev1982,
    AUTHOR = {Tankeev, S. G.},
     TITLE = {Cycles on simple abelian varieties of prime dimension},
   JOURNAL = {Izv. Akad. Nauk SSSR Ser. Mat.},
  FJOURNAL = {Izvestiya Akademii Nauk SSSR. Seriya Matematicheskaya},
    VOLUME = {46},
      YEAR = {1982},
    NUMBER = {1},
     PAGES = {155--170, 192},
      ISSN = {0373-2436},
   MRCLASS = {14G13 (14K99)},
  MRNUMBER = {643899},
MRREVIEWER = {Yu.\ G.\ Zarkhin},
}

@article {Milne1999,
    AUTHOR = {Milne, J. S.},
     TITLE = {Lefschetz classes on abelian varieties},
   JOURNAL = {Duke Math. J.},
  FJOURNAL = {Duke Mathematical Journal},
    VOLUME = {96},
      YEAR = {1999},
    NUMBER = {3},
     PAGES = {639--675},
      ISSN = {0012-7094,1547-7398},
   MRCLASS = {14C25 (11G10 14C30 14K05)},
  MRNUMBER = {1671217},
MRREVIEWER = {Fumio\ Hazama},
       DOI = {10.1215/S0012-7094-99-09620-5},
       URL = {https://doi.org/10.1215/S0012-7094-99-09620-5},
}

@misc{Markman2026,
    title={Secant sheaves and Weil classes on abelian varieties}, 
    author={Eyal Markman},
    year={2026},
    version={v2},
    eprint={2509.23403},
    archivePrefix={arXiv},
    primaryClass={math.AG},
    url={https://arxiv.org/abs/2509.23403}, 
}

@inbook{Weil1979,
  author       = {Weil, Andr{\'e}},
  title        = {Abelian varieties and the {H}odge ring},
  bookauthor   = {Weil, Andr{\'e}},
  booktitle    = {{\OE}uvres scientifiques / Collected Papers},
  volume       = {3},
  pages        = {421--429},
  publisher    = {Springer-Verlag},
  location     = {New York},
  year         = {1979},
  note         = {Original paper dated 1977; corrected second printing, 1980}
}

@online{stacks-project,
  author       = {{{The Stacks Project Authors}}},
  title        = {The Stacks Project},
  year         = {2026},
  shorthand    = {Sta26},
  url          = {https://stacks.math.columbia.edu},
  urldate      = {2026-08-06}
}

@article {Kahn2003,
    AUTHOR = {Kahn, Bruno},
     TITLE = {\'Equivalences rationnelle et num\'erique sur certaines
              vari\'et\'es de type ab\'elien sur un corps fini},
   JOURNAL = {Ann. Sci. \'Ecole Norm. Sup. (4)},
  FJOURNAL = {Annales Scientifiques de l'\'Ecole Normale Sup\'erieure.
              Quatri\`eme S\'erie},
    VOLUME = {36},
      YEAR = {2003},
    NUMBER = {6},
     PAGES = {977--1002},
      ISSN = {0012-9593},
   MRCLASS = {14G15 (11G10 11G25 14C25 14F42 19E15)},
  MRNUMBER = {2032532},
MRREVIEWER = {R.\ T.\ Hoobler},
       DOI = {10.1016/j.ansens.2003.02.002},
       URL = {https://doi.org/10.1016/j.ansens.2003.02.002},
}

@incollection {Kleiman1968,
    AUTHOR = {Kleiman, S. L.},
     TITLE = {Algebraic cycles and the {W}eil conjectures},
 BOOKTITLE = {Dix expos\'es sur la cohomologie des sch\'emas},
    SERIES = {Adv. Stud. Pure Math.},
    VOLUME = {3},
     PAGES = {359--386},
 PUBLISHER = {North-Holland, Amsterdam},
      YEAR = {1968},
   MRCLASS = {14C25 (14F20)},
  MRNUMBER = {292838},
MRREVIEWER = {H.\ Popp},
}

@incollection {Grothendieck1969,
    AUTHOR = {Grothendieck, A.},
     TITLE = {Standard conjectures on algebraic cycles},
 BOOKTITLE = {Algebraic {G}eometry ({I}nternat. {C}olloq., {T}ata {I}nst.
              {F}und. {R}es., {B}ombay, 1968)},
    SERIES = {Tata Inst. Fundam. Res. Stud. Math.},
    VOLUME = {4},
     PAGES = {193--199},
 PUBLISHER = {Tata Inst. Fund. Res., Bombay},
      YEAR = {1969},
   MRCLASS = {14.40},
  MRNUMBER = {268189},
MRREVIEWER = {S.\ L.\ Kleiman},
}

@misc{Milne2019,
    AUTHOR = {Milne, J. S.},
    TITLE = {On the {T}ate and standard conjectures over finite fields},
    YEAR = {2019},
    VERSION={v1.1},
    URL = {https://www.jmilne.org/math/articles/TFF.pdf},
}

@article {Clozel1999,
    AUTHOR = {Clozel, L.},
     TITLE = {Equivalence num\'erique et \'equivalence cohomologique pour
              les vari\'et\'es ab\'eliennes sur les corps finis},
   JOURNAL = {Ann. of Math. (2)},
  FJOURNAL = {Annals of Mathematics. Second Series},
    VOLUME = {150},
      YEAR = {1999},
    NUMBER = {1},
     PAGES = {151--163},
      ISSN = {0003-486X,1939-8980},
   MRCLASS = {14C25 (14F20 14K15 14K22)},
  MRNUMBER = {1715322},
MRREVIEWER = {Amnon\ Besser},
       DOI = {10.2307/121099},
       URL = {https://doi.org/10.2307/121099},
}

@misc{Kelly2026,
      title={Some explicit counter-examples to Weibel's conjecture}, 
      author={Shane Kelly},
      year={2026},
      version={v1},
      eprint={2608.16066},
      archivePrefix={arXiv},
      primaryClass={math.AG},
      url={https://arxiv.org/abs/2608.16066}, 
}

@article {Agugliaro2026,
    AUTHOR = {Agugliaro, Thomas},
     TITLE = {Examples for the standard conjecture of {H}odge type},
   JOURNAL = {Doc. Math.},
  FJOURNAL = {Documenta Mathematica},
    VOLUME = {31},
      YEAR = {2026},
    NUMBER = {4},
     PAGES = {855--904},
      ISSN = {1431-0635,1431-0643},
   MRCLASS = {14C15 (11R04)},
  MRNUMBER = {5081020},
       DOI = {10.4171/dm/1057},
       URL = {https://doi.org/10.4171/dm/1057},
}
\end{document}